\documentclass{amsart}
\usepackage{amssymb,amsmath, amsthm,latexsym}
\usepackage{psfrag}
\usepackage{amscd}
\usepackage{graphicx}
\newcommand{\cal}[1]{\mathcal{#1}}
\theoremstyle{plain}
\newtheorem*{theo}{Theorem}

\newtheorem{lemma}{Lemma}[section]
\newtheorem{theorem}[lemma]{Theorem}
\newtheorem{proposition}[lemma]{Proposition}
\newtheorem{corollary}[lemma]{Corollary}
\theoremstyle{definition}
\newtheorem{definition}[lemma]{Definition}
\parskip=\bigskipamount

\let\egthree=\phi
\let\phi=\varphi
\let\varphi=\egthree

\begin{document}
\title{Bowen's construction for the Teichm\"uller flow}
\author{Ursula Hamenst\"adt}
\thanks
{
Research
partially supported by a grant of the DFG}
%\dedicatory{To the memory of Martine Babillot}
\date{December 22, 2010}
\begin{abstract}
Let ${\cal Q}$ be a connected component of a stratum
in the moduli space of abelian or quadratic differentials for 
a non-exceptional
Riemann surface $S$ of finite type. We show that the probability measure
on ${\cal Q}$ 
in the Lebesgue measure class which is invariant under
the Teichm\"uller flow is obtained by
Bowen's construction.
\end{abstract}

\maketitle

\section{Introduction}

The \emph{Teichm\"uller flow} $\Phi^t$ 
acts on components of 
strata in the moduli space of area one abelian or 
quadratic differentials for a non-exceptional surface $S$ of finite type.
This flow has many properties which resemble
the properties of an Anosov flow. For example, there is
a pair of transverse invariant foliations, and there is 
an invariant mixing 
Borel probability measure $\lambda$ in the Lebesgue measure class
which is absolutely continuous with respect to these foliations, with  
conditional measures which are uniformly
expanded and contracted by the flow \cite{M82,V86}.  This measure
is even exponentially mixing, i.e. 
exponential decay of correlations for H\"older observables
holds true \cite{AGY06,AR09}.

The entropy $h$ of the Lebesgue measure $\lambda$
is the supremum of the topological entropies of the restriction 
of $\Phi^t$ to compact invariant sets \cite{H07b}. 
For strata of abelian differentials,
$\lambda$ is the unique
invariant measure of maximal entropy \cite{BG07}.

The goal of this note is to extend further the analogy between
the Teichm\"uller flow on components of strata 
and Anosov flows.  
An Anosov flow $\Psi^t$ on a compact 
manifold $M$ admits a unique Borel
probability measure $\mu$ of maximal entropy.
This measure can be
obtained as follows \cite{B73}. Every periodic 
orbit $\gamma$ of $\Psi^t$ 
of prime period $\ell(\gamma)>0$ supports  
a unique $\Psi^t$-invariant Borel measure $\delta(\gamma)$ 
of total mass $\ell(\gamma)$. If $h>0$ is the topological
entropy of $\Psi^t$ then $\mu$ 
is the (unique) weak limit  
of the sequence of measures 
\[e^{-hR}\sum_{\ell(\gamma)\leq R}\delta(\gamma)\]
as $R\to \infty$. 
In particular, the number of periodic orbits 
of period at most $R$ is
asymptotic to $e^{hR}/hR$ as $R\to \infty$.

For any connected component ${\cal Q}$ of a stratum
of abelian or quadratic differentials 
the $\Phi^t$-invariant 
Lebesgue measure $\lambda$ on ${\cal Q}$
can be obtained in the same way. 
For a precise formulation, we say that a family $\{\mu_i\}$ of 
finite Borel measures on the moduli space
${\cal H}(S)$ of area one abelian differentials
or on the moduli space 
${\cal Q}(S)$ of area one quadratic differentials
\emph{converges weakly} to $\lambda$ if for every continuous
function $f$ on ${\cal H}(S)$ or on ${\cal Q}(S)$ 
with compact support we have 
\[\int fd\mu_i\to \int fd\lambda.\]
Let $\Gamma({\cal Q})$ be the
set of all periodic orbits for $\Phi^t$ contained in
${\cal Q}$. 
For $\gamma\in \Gamma({\cal Q})$ let 
$\ell(\gamma)>0$ be the prime period of $\gamma$ and denote
by $\delta(\gamma)$ the $\Phi^t$-invariant
Lebesgue measure on $\gamma$ of total mass $\ell(\gamma)$.
We show

\begin{theo} For every component ${\cal Q}$ of a stratum 
in the moduli space of abelian or quadratic differentials
 the measures
\[\mu_R=e^{-hR}\sum_{\gamma\in 
\Gamma({\cal Q}),\ell(\gamma)\leq R}\delta(\gamma)\]
converge as $R\to \infty$ weakly 
to the Lebesgue measure on ${\cal Q}$.
\end{theo}

The theorem implies that as $R\to \infty$, the
number of periodic orbits in ${\cal Q}$
of period at most $R$ 
is asymptotically not
smaller than  $e^{hR}/hR$. However,  
since the closure in ${\cal Q}(S)$ of a 
component ${\cal Q}$ of a stratum is non-compact, 
we do not obtain a precise asymptotic growth rate 
for all periodic orbits in ${\cal Q}$. 
Namely, there may be a set of periodic orbits in ${\cal Q}$ 
whose growth rate
exceeds $h$ and which eventually exit every compact subset of 
${\cal Q}(S)$. For periodic orbits in the open principal stratum, 
Eskin and Mirzakhani \cite{EM07} showed
that the asymptotic growth rate
of periodic orbits for the Teichm\"uller flow  
which lie deeply in the cusp of moduli space 
is strictly smaller than the entropy $h$, and they 
calculate the asymptotic growth rate of all periodic orbits. 
Eskin, Mirzakhani and Rafi \cite{EMR10} 
also announced the
analogous result  for any component of any stratum.

The proof of the above theorem uses ideas which were developed 
by Margulis for hyperbolic flows (see \cite{Mar04} for 
an account with comments). This strategy is by now standard,
and the main task is to overcome the 
difficulty of absence of 
hyperbolicity for the Teichm\"uller flow 
in the thin part of moduli space and the 
absence of nice product coordinates near a
boundary point of a stratum.

Absence of hyperbolicity in the thin part of moduli space is
dealt with using the curve graph similar to the strategy 
developed in \cite{H07b}. Integration of the Hodge norm
as discussed in \cite{ABEM10} and some standard ergodic theory
is also used.

Relative homology coordinates \cite{V90} define
local product structures for
strata. These coordinates 
do no extend in 
a straightforward way to points in the boundary of the stratum. 
In the case of the principal stratum, however, 
product coordinates about boundary points can 
be obtained by simply writing a quadratic differential
as a pair of its vertical and horizontal measured geodesic 
lamination. 
Our approach is to show that
there is a similar picture for strata. To this end, we use coordinates
for strata based on train tracks which will be
used in other contexts as well. The construction of these coordinates
is carried out in Sections 3 and 4.

The tools developed in Sections 3 and 4 are used in Section 5 to show that
a weak limit $\mu$ of the measures $\mu_R$ is 
absolutely continuous with respect to the Lebesgue measure, with
Radon Nikodym derivative bounded from above by one. 
In Section 6 the proof of the theorem is completed.
Section 2 summarizes some properties of the curve graph and
geodesic laminations used throughout the paper.

\section{Laminations and the curve graph}

Let $S$ be an oriented surface of finite type, i.e. $S$ is a
closed surface of genus $g\geq 0$ from which 
$m\geq 0$ points, so-called \emph{punctures}, have been deleted.
We assume that $3g-3+m\geq 2$,
i.e. that $S$ is not a sphere with at most four
punctures or a torus with at most one puncture.
The \emph{Teichm\"uller space} ${\cal T}(S)$
of $S$ is the quotient of the space of all complete finite volume hyperbolic
metrics on $S$ under the action of the
group of diffeomorphisms of $S$ which are isotopic
to the identity. The
fibre bundle ${\cal Q}^1(S)$
over ${\cal T}(S)$ of all \emph{marked holomorphic
quadratic differentials} of area
one can be viewed as the unit cotangent
bundle of ${\cal T}(S)$ for the \emph{Teichm\"uller metric} $d_{\cal T}$.
We assume that each quadratic differential $q\in {\cal Q}^1(S)$
has a pole of first order at each of the punctures, i.e. we
include the information on the number of poles of the
differential in the number of punctures of $S$.
The \emph{Teichm\"uller flow} $\Phi^t$ on
${\cal Q}^1(S)$ commutes
with the action of the \emph{mapping class group}
${\rm Mod}(S)$ of all isotopy classes of
orientation preserving self-homeomorphisms of $S$.
Therefore this flow descends to a flow
on the quotient orbifold ${\cal Q}(S)=
{\cal Q}^1(S)/{\rm Mod}(S)$, again denoted
by $\Phi^t$. 

\subsection{Geodesic laminations}

A \emph{geodesic lamination} for a complete
hyperbolic structure on $S$ of finite volume is
a \emph{compact} subset of $S$ which is foliated into simple
geodesics.
A geodesic lamination $\nu$ is called \emph{minimal}
if each of its half-leaves is dense in $\nu$. Thus a simple
closed geodesic is a minimal geodesic lamination. A minimal
geodesic lamination with more than one leaf has uncountably
many leaves and is called \emph{minimal arational}.
Every geodesic lamination $\nu$ consists of a disjoint union of
finitely many minimal components and a finite number of isolated
leaves. Each of the isolated leaves of $\nu$ either is an
isolated closed geodesic and hence a minimal component, or it
\emph{spirals} about one or two minimal components.
A geodesic lamination $\nu$ \emph{fills up} $S$ if its complementary
components are topological discs or once punctured monogons, i.e.
once punctured discs bounded by a single leaf of $\nu$.

The set ${\cal L}$ of all geodesic laminations
on $S$ can be equipped with the restriction of 
the \emph{Hausdorff topology} for compact subsets of $S$.
With respect to this topology, the space ${\cal L}$ is
compact. 

The projectivized tangent bundle $PT\nu$ 
of a geodesic lamination $\nu$ 
is a compact subset of the projectivized tangent
bundle $PTS$ of $S$. 
The geodesic lamination $\nu$ is \emph{orientable} if
there is an continuous 
orientation of the tangent bundle of $\nu$. This is
equivalent to stating that there is a continuous section
$PT\nu\to T^1S$ where 
$T^1S$ denotes the unit tangent bundle of $S$.

\begin{definition}\label{largelam}
A \emph{large geodesic lamination} is a geodesic
lamination  
$\nu$ which fills
up $S$ and can be approximated in the
Hausdorff topology by simple closed geodesics. 
\end{definition}

Note that a minimal geodesic lamination $\nu$ can
be approximated in the Hausdorff topology by simple
closed geodesics and hence if $\nu$ fills up $S$ then
$\nu$ is large. Moreover, the set of all large 
geodesic laminations is closed with respect to the
Hausdorff topology and hence it is compact.

The 
\emph{topological type} of a large geodesic lamination
$\nu$ is a tuple 
\[(m_1,\dots,m_\ell;-m)\text{ where }
1\leq m_1\leq \dots \leq m_{\ell},\, \sum_im_i=4g-4+m\] 
such that the complementary
components of $\nu$ which are topological discs are $m_i+2$-gons. 
Let \[{\cal L\cal L}(m_1,\dots,m_\ell;-m)\] be the space of all
large geodesic laminations of type $(m_1,\dots,m_\ell;-m)$ equipped 
with the restriction of the Hausdorff topology for compact
subsets of $S$. A geodesic
lamination is called \emph{complete} if it  is large
of type $(1,\dots,1;-m)$. The
complementary components of a complete geodesic
lamination are all trigons or
once punctured monogons.

A \emph{measured geodesic lamination} is 
a geodesic lamination $\nu$ equipped with a
translation invariant transverse measure $\xi$ such that
the $\xi$-weight of every compact
arc in $S$ with endpoints in $S-\nu$ 
which intersects $\nu$ nontrivially
and transversely is positive. We say
that $\nu$ is the \emph{support} of the 
measured geodesic lamination. 
The geodesic lamination $\nu$ 
is \emph{uniquely ergodic} if $\xi$ is the
only transverse measure with support $\nu$ up to
scale.

The space ${\cal M\cal L}$
of measured geodesic laminations 
equipped with the weak$^*$-topology admits a
natural continuous action of the multiplicative group
$(0,\infty)$. The quotient under this action is
the space ${\cal P\cal M\cal L}$ of 
\emph{projective measured geodesic laminations}
which is homeomorphic to the sphere $S^{6g-7+2m}$.

Every simple closed geodesic $c$ on $S$ defines
a measured geodesic lamination. The geometric
intersection number between simple closed
curves on $S$ extends to a continuous 
function $\iota$ on ${\cal M\cal L}\times {\cal M\cal L}$, the
\emph{intersection form}. We say that
a pair $(\xi,\mu)\in {\cal M\cal L}\times
{\cal M\cal L}$ of measured geodesic laminations
\emph{jointly fills up} $S$ if 
for every measured geodesic
lamination $\eta\in {\cal M\cal L}$ we have
$\iota(\eta,\xi)+\iota(\eta,\mu)>0$. This
is equivalent to stating that every complete 
simple (possibly infinite)
geodesic on $S$ intersects either the support
of $\xi$ or the support of $\mu$ transversely.

\subsection{The curve graph}

The \emph{curve graph} ${\cal C}(S)$ of $S$ is the locally
infinite metric graph whose vertices are the free homotopy classes
of essential simple closed curves on $S$, i.e. curves which
are neither contractible nor freely homotopic into a puncture. 
Two such curves
are connected by an edge of length one if and only if they
can be realized disjointly. The mapping class group 
${\rm Mod}(S)$ of $S$ acts on ${\cal C}(S)$ as a group of 
simplicial isometries.

The curve graph ${\cal C}(S)$ is a hyperbolic geodesic
metric space \cite{MM99} and hence it admits
a \emph{Gromov boundary} $\partial {\cal C}(S)$.
For $c\in {\cal C}(S)$ there is a complete distance
function $\delta_c$ on $\partial{\cal C}(S)$ of 
uniformly bounded diameter, and there is
a number $\rho >0$ such that 
\[\delta_c\leq e^{\rho
d(c,a)}\delta_a\text{ for all }c,a\in {\cal C}(S).\] 
The group ${\rm Mod}(S)$ acts on 
$\partial {\cal C}(S)$ as a group of homeomorphisms.

Let $\kappa_0>0$ be a \emph{Bers constant} for $S$,
i.e. $\kappa_0$ is such that for every complete
hyperbolic metric on $S$ of finite volume 
there is a pants decomposition of $S$
consisting of pants curves of length at most $\kappa_0$.
Define a map 
\begin{equation}\label{upsilont}
\Upsilon_{\cal T}:{\cal T}(S)\to {\cal C}(S)\end{equation}
by associating to $x\in {\cal T}(S)$ 
a simple closed 
curve of $x$-length at most $\kappa_0$.
Then there is a number $c>0$ such that
\begin{equation}\label{distortion}
d_{\cal T}(x,y)\geq d(\Upsilon_{\cal T}(x),\Upsilon_{\cal T}(y))/c-c
\end{equation}
for all $x,y\in {\cal T}(S)$ 
(see the discussion in \cite{H10a}).

For a number $L>1$, 
a map $\gamma:[0,s)\to {\cal C}(S)$ $(s\in (0,\infty])$ is an
\emph{$L$-quasi-geodesic} if 
for all
$t_1,t_2\in [0,s)$ we have
\[\vert t_1-t_2\vert/L-L\leq d(\gamma(t_1),\gamma(t_2))\leq
L\vert t_1-t_2\vert +L.\]
A map $\gamma:[0,\infty)\to {\cal C}(S)$ is called
an \emph{unparametrized $L$-quasi-geodesic} if 
there is an increasing homeomorphism
$\phi:[0,s)\to [0,\infty)$ $(s\in (0,\infty])$ 
such that $\gamma\circ \phi$
is an $L$-quasi-geodesic. 
We say that an unparametrized
quasi-geodesic is \emph{infinite} 
if its image set has infinite diameter. There is a number $p>1$ such that the
image under $\Upsilon_{\cal T}$ of every Teichm\"uller geodesic 
is an unparametrized $p$-quasi-geodesic \cite{MM99}.

Choose
a smooth function $\sigma:[0,\infty)\to [0,1]$ with
$\sigma[0,\kappa_0]\equiv 1$ and
$\sigma[2\kappa_0,\infty)\equiv 0$. For each
$x\in {\cal T}(S)$, the number of essential
simple closed curves $c$
on $S$ whose $x$-length
$\ell_x(c)$ (i.e. the length of a geodesic representative
in its free homotopy class)
does not exceed $2\kappa_0$ is bounded from above by a
constant not depending on $x$, and the diameter of the
subset of ${\cal C}(S)$ containing these curves is uniformly
bounded as well. Thus we obtain for every
$x\in {\cal T}(S)$ a
finite Borel measure $\mu_x$ on ${\cal C}(S)$ by defining
\begin{equation}
\mu_x=\sum_{c\in {\cal C}(S)} \sigma(\ell_x(c))\Delta_c \notag
\end{equation} where
$\Delta_c$ denotes the Dirac mass at $c$. The total mass of
$\mu_x$ is bounded from above and below by a universal positive
constant, and the diameter of the support of $\mu_x$ in ${\cal
C}(S)$ is uniformly bounded as well. Moreover, the measures
$\mu_x$ depend continuously on $x\in {\cal T}(S)$ in the
weak$^*$-topology. This means that for every bounded function
$f:{\cal C}(S)\to \mathbb{R}$ the function $x\to \int f d\mu_x$ is
continuous.

For $x\in {\cal
T}(S)$ define a distance $\delta_x$ on $\partial {\cal C}(S)$ by
\begin{equation}\label{distance}
\delta_x(\xi,\zeta)=\int \delta_c(\xi,\zeta)
d\mu_x(c)/\mu_x({\cal C}(S)).\end{equation}
The distances
$\delta_x$ are equivariant with respect to the
action of ${\rm Mod}(S)$ on ${\cal T}(S)$ and
$\partial{\cal C}(S)$. Moreover, there is a constant
$\kappa >0$ such that
\begin{equation}\label{deltacomparison}
\delta_x\leq e^{\kappa d_{\cal T}(x,y)}\delta_y\text{ and }
\kappa^{-1}\delta_y\leq 
\delta_{\Upsilon_{\cal T}(y)}\leq \kappa \delta_y\end{equation}
for all $x,y\in {\cal T}(S)$ (see p.230 and p.231 of \cite{H09}).

An area one quadratic differential 
$z\in {\cal Q}^1(S)$ is determined by a pair $(\mu,\nu)$ of 
measured geodesic laminations which jointly fill up $S$ and such that
$\iota(\mu,\nu)=1$. The laminations $\mu,\nu$ are called 
\emph{vertical} and \emph{horizontal}, respectively.
For $z\in {\cal Q}^1(S)$ let $W^{u}(z)\subset {\cal Q}^1(S)$ 
be the set of all quadratic
differentials whose horizontal projective measured geodesic lamination 
coincides with the horizontal projective measured geodesic lamination of $z$.
The space $W^u(z)$ is called
the \emph{unstable} manifold of $z$, and these unstable
manifolds define the \emph{unstable foliation} $W^u$ of ${\cal Q}^1(S)$. 
The \emph{strong unstable manifold} $W^{su}(z)\subset
W^u(z)$ is the set of all quadratic differentials whose
horizontal measured geodesic lamination coincides with the horizontal
measured geodesic lamination of $z$. 
These sets define the \emph{strong unstable} 
foliation $W^{su}$ of ${\cal Q}^1(S)$.
The image of the unstable (or the strong unstable) foliation
of ${\cal Q}^1(S)$ under the flip ${\cal F}:q\to {\cal F}(q)=-q$
is the \emph{stable} foliation $W^s$ (or the 
\emph{strong stable foliation} $W^{ss}$).

By the Hubbard-Masur theorem, for each $z\in {\cal Q}^1(S)$ 
the restriction to $W^{u}(z)$ 
of the canonical projection
\[P:{\cal Q}^1(S)\to {\cal T}(S)\]
is a homeomorphism. Thus the Teichm\"uller metric lifts to 
a complete distance function $d^u$ on $W^u(z)$. Denote by $d^{su}$ 
the restriction of this distance function to $W^{su}(z)$.
Then $d^s=d^u\circ {\cal F},d^{ss}=d^{su}\circ {\cal F}$
are distance functions on the leaves of the stable and
strong stable foliation, respectively.
For $z\in {\cal Q}^1(S)$ and $r>0$ let moreover
$B^i(z,r)\subset W^i(z)$ be the closed ball of radius $r$ about $z$ 
with respect to $d^i$ $(i=u,su,s,ss)$.

Let 
\begin{equation}\label{tildea}
\tilde {\cal A}\subset {\cal Q}^1(S)\end{equation} 
be the
set of all marked quadratic differentials
$q$ such that the unparametrized quasi-geodesic
$t\to \Upsilon_{\cal T}(P\Phi^tq)$ $(t\in [0,\infty))$ is infinite.
Then $\tilde {\cal A}$
is the set of all quadratic differentials
whose vertical measured geodesic lamination fills up $S$
(i.e. its support fills up $S$, 
see \cite{H06} for a comprehensive discussion of this result
of Klarreich \cite{Kl99}).
There is a natural ${\rm Mod}(S)$-equivariant 
surjective map 
\[F:\tilde {\cal A}\to
\partial {\cal C}(S)\] which associates to a point
$z\in \tilde{\cal A}$ 
the endpoint of the infinite unparametrized quasi-geodesic
$t\to \Upsilon_{\cal T}(P\Phi^tq)$ $(t\in [0,\infty))$.
 
Call a marked quadratic differential $z\in {\cal Q}^1(S)$ 
\emph{uniquely ergodic} if the support of its vertical measured geodesic
lamination is uniquely ergodic and fills up $S$.
A uniquely ergodic quadratic differential is contained
in the set $\tilde{\cal A}$ \cite{H06,Kl99}.
 We have (Section 3 of \cite{H09})

\begin{lemma}\label{closed}
\begin{enumerate}
\item
The map $F:\tilde {\cal A}\to \partial{\cal C}(S)$ is 
continuous and closed.
\item If $z\in {\cal Q}^1(S)$ is uniquely ergodic then the sets
$F(B^{su}(z,r)\cap \tilde{\cal A})$ $(r>0)$ form a neighborhood basis for 
$F(z)$ in $\partial{\cal C}(S)$.
\end{enumerate}
\end{lemma}

For $z\in \tilde {\cal A}$ and $r>0$ let 
\[D(z,r)\]
be the closed ball of radius $r$ about $F(z)$ with respect to the
distance function $\delta_{Pz}$. 
As a consequence of Lemma \ref{closed}, if $z\in 
{\cal Q}^1(S)$ is uniquely ergodic then for every $r>0$
there are numbers $r_0<r$ and $\beta >0$ such that
\begin{equation}F(B^{su}(z,r_0)\cap \tilde {\cal A})
\subset D(z,\beta)
\subset
F(B^{su}(z,r)\cap\tilde {\cal A}).\end{equation}

\section{Train tracks}

In this section we establish some properties of train tracks
on  an oriented surface $S$ of
genus $g\geq 0$ with $m\geq 0$ punctures and $3g-3+m\geq 2$
which will be used in Section 4 
to construct coordinates near boundary points
of strata.

A \emph{train track} on $S$ is an embedded
1-complex $\tau\subset S$ whose edges
(called \emph{branches}) are smooth arcs with
well-defined tangent vectors at the endpoints. At any vertex
(called a \emph{switch}) the incident edges are mutually tangent.
Through each switch there is a path of class $C^1$
which is embedded
in $\tau$ and contains the switch in its interior. 
A simple closed curve component of $\tau$ contains
a unique bivalent switch, and all other switches are at least
trivalent.
The complementary regions of the
train track have negative Euler characteristic, which means
that they are different from discs with $0,1$ or
$2$ cusps at the boundary and different from
annuli and once-punctured discs
with no cusps at the boundary.
We always identify train
tracks which are isotopic.
Throughout we use the book \cite{PH92} as the main reference for 
train tracks. 

A train track is called \emph{generic} if all switches are
at most trivalent. For each switch $v$ of 
a generic train track $\tau$ which is not contained in 
a simple closed curve component, there is a unique
half-branch $b$ of $\tau$ which is incident on $v$ and which is
\emph{large} at $v$. This means that every germ of an
arc of class $C^1$ on $\tau$ which passes through $v$ also
passes through the interior of $b$. 
A half-branch which is not large is called \emph{small}.
A branch $b$ of $\tau$ is
called \emph{large} (or \emph{small}) if each of its
two half-branches is large (or small). A branch which 
is neither large nor small is called \emph{mixed}.

\bigskip
{\bf Remark:} As in \cite{H09a}, all train tracks
are assumed to be generic. Unfortunately this leads to 
a small inconsistency of our terminology with the
terminology found in the literature.

\bigskip

A \emph{trainpath} on a train track $\tau$ is a
$C^1$-immersion $\rho:[k,\ell]\to \tau$ such that
for every $i< \ell-k$ the restriction of $\rho$ to 
$[k+i,k+i+1]$ is a homeomorphism onto a branch of $\tau$.
More generally, we call a $C^1$-immersion $\rho:[a,b]\to \tau$
a \emph{generalized trainpath}. 
A trainpath $\rho:[k,\ell]\to \tau$ is \emph{closed}
if $\rho(k)=\rho(\ell)$ and if the extension $\rho^\prime$
defined by $\rho^\prime(t)=\rho(t)$ 
$(t\in [k,\ell])$ and 
$\rho^\prime(\ell+s)=\rho(k+s)$ $(s\in [0,1])$ is a
trainpath.

A generic 
train track $\tau$ is \emph{orientable} 
if there is a consistent orientation of the 
branches of $\tau$ such that 
at any switch $s$ of $\tau$, the orientation of the large
half-branch incident on $s$ extends to the orientation
of the two small half-branches incident on $s$.
If $C$ is a complementary polygon of an oriented
train track then the number of sides of $C$ is even.
In particular, a train track which contains a once
punctured monogon component,  
i.e. a once punctured disc with
one cusp at the boundary, is not orientable
(see p.31 of \cite{PH92} for 
a more detailed discussion).

A train track or a geodesic lamination $\eta$ is
\emph{carried} by a train track $\tau$ if
there is a map $F:S\to S$ of class $C^1$ which is homotopic to the
identity and maps $\eta$ into $\tau$ in such a way 
that the restriction of the differential of $F$
to the tangent space of $\eta$ vanishes nowhere;
note that this makes sense since a train track has a tangent
line everywhere. We call the restriction of $F$ to
$\eta$ a \emph{carrying map} for $\eta$.
Write $\eta\prec
\tau$ if the train track $\eta$ is carried by the train track
$\tau$. Then every geodesic lamination $\nu$ which is carried
by $\eta$ is also carried by $\tau$.

A train track \emph{fills up} $S$ if its complementary
components are topological discs or once punctured 
monogons.  Note that such a train track
$\tau$ is connected.
Let $\ell\geq 1$ be the number of those complementary 
components of $\tau$ which are topological discs.
Each of these discs is an $m_i+2$-gon for some $m_i\geq 1$
$(i=1,\dots,\ell)$. The
\emph{topological type} of $\tau$ is defined to be
the ordered tuple $(m_1,\dots,m_\ell;-m)$ where
$1\leq m_1\leq \dots \leq m_\ell$; then $\sum_im_i=4g-4+m$.
If $\tau$ is orientable then $m=0$ and $m_i$ is even 
for all $i$. 
A train track of topological type $(1,\dots,1;-m)$ is called 
\emph{maximal}. The complementary components
of a maximal train track are all trigons,
i.e. topological discs with three cusps at the boundary,
or once punctured monogons.

A \emph{transverse measure} on a generic train track $\tau$ is a
nonnegative weight function $\mu$ on the branches of $\tau$
satisfying the \emph{switch condition}:
for every trivalent switch $s$ of $\tau$,  the sum of the weights
of the two small half-branches incident on $s$ 
equals the weight of the large half-branch.
The space ${\cal V}(\tau)$ of all transverse measures
on $\tau$ has the structure of a cone in a finite dimensional
real vector space, and
it is naturally homeomorphic to the
space of all measured geodesic laminations whose
support is carried by $\tau$.
The train track is called
\emph{recurrent} if it admits a transverse measure which is
positive on every branch. We call such a transverse measure $\mu$
\emph{positive}, and we write $\mu>0$ (see \cite{PH92} for 
more details).

A \emph{subtrack} $\sigma$ of a train track $\tau$
is a subset of $\tau$ which is itself a train track.
Then $\sigma$ is obtained from $\tau$ by removing some
of the branches, and 
we write $\sigma <\tau$.
If $b$ is a small branch of $\tau$ which is incident on two
distinct switches of $\tau$ then
the graph $\sigma$
obtained from $\tau$ by removing $b$ is a subtrack of $\tau$.
We then call $\tau$ a 
\emph{simple extension} of $\sigma$.
Note that formally to obtain the subtrack $\sigma $ from $\tau-b$
we may have to delete the switches on which the branch $b$ is incident. 

\begin{lemma} \label{simpleex}
\begin{enumerate}
\item
A simple extension $\tau$ of 
a recurrent non-orientable connected train track $\sigma$
is recurrent. Moreover,
\[{\rm dim}{\cal V}(\sigma)={\rm dim}{\cal V}(\tau)-1.\]
\item An orientable simple extension $\tau$ of a recurrent orientable 
connected train track $\sigma$
is recurrent. Moreover,
\[{\rm dim}{\cal V}(\sigma)={\rm dim}{\cal V}(\tau)-1.\]
\end{enumerate}
\end{lemma}
\begin{proof}
If $\tau$ is a simple extension of a 
train track $\sigma$ then $\sigma$ can be
obtained from $\tau$ by the removal of a small branch $b$ which
is incident on two distinct switches
$s_1,s_2$. Then $s_i$ is an interior point of a branch
$b_i$ of $\sigma$ $(i=1,2)$. 

If $\sigma$ is connected, non-orientable and recurrent then there is 
a trainpath $\rho_0:[0,t]\to 
\tau-b$ which begins at $s_1$, ends at $s_2$ and such that
the half-branch $\rho_0[0,1/2]$ is small at $s_1=\rho_0(0)$ and  
that the half-branch $\rho_0[t-1/2,t]$ is small at $s_2=\rho_0(t)$. 
Extend $\rho_0$ to a closed trainpath $\rho$ on $\tau -b$ 
which begins and ends at $s_1$. This is possible since
$\sigma$ is non-orientable, connected and recurrent.
There is a closed trainpath $\rho^\prime:[0,u]\to \tau$ 
which can be obtained from $\rho$ by replacing the trainpath $\rho_0$ by
the branch $b$ traveled through
from $s_1$ to $s_2$. The counting measure of $\rho^\prime$ on $\tau$ 
satisfies the switch condition and hence it
defines a transverse measure on $\tau$ which is positive on $b$.
On the other hand, every transverse measure on $\sigma$
defines a transverse measure on $\tau$. Thus 
since $\sigma$ is recurrent and since the 
sum of two transverse measures on $\tau$ is 
again a transverse
measure,  the train track $\tau$ is recurrent as well.
Moreover, we have 
${\rm dim}{\cal V}(\tau)\geq {\rm dim}{\cal V}(\sigma)+1$.

Let
$p$ be the number of branches of $\tau$.
Label the branches of $\tau$ 
with the numbers $\{1,\dots,p\}$ so that the number $p$ is 
assigned to $b$. 
Let $e_1,\dots,e_p$ be the standard basis of $\mathbb{R}^p$ and
define a linear
map $A:\mathbb{R}^p\to \mathbb{R}^p$ by 
$A(e_i)=e_i$ for $i\leq p-1$ and 
$A(e_p)=\sum_i\nu(i)e_i$ where
$\nu$ is the weight function on $\{1,\dots,p\}$ defined by the
trainpath $\rho_0$. The map $A$ is a surjection onto a 
linear subspace of $\mathbb{R}^p$ of codimension one, moreover
$A$ preserves the linear subspace $V$ of $\mathbb{R}^p$ defined
by the switch conditions for $\tau$. In particular,
the corank of $A(V)$ is at most one.
The image under $A$ of the cone of all 
nonnegative weights
on the branches of $\tau$ satisfying the switch conditions is
contained in the cone of all nonnegative weights on $\tau-b=\sigma$ 
satisfying the switch
conditions for $\sigma$. Therefore the
dimension of the space of transverse measures on 
$\sigma$ equals the space of transverse measures on 
$\tau$ minus one. This implies ${\rm dim}({\cal V}(\tau)=
{\rm dim}({\cal V}(\sigma)+1$ and 
completes the proof of the first part of 
the lemma. The second part follows in exactly the same way.
\end{proof}

As a consequence we obtain

\begin{corollary}\label{dimensioncount}
\begin{enumerate}
\item
${\rm dim}{\cal V}(\tau)= 2g-2+m+\ell$
for every non-orientable recurrent
train track $\tau$ of topological type $(m_1,\dots,m_\ell;-m)$.
\item ${\rm dim}{\cal V}(\tau)= 2g-1+\ell$ for every orientable
recurrent train track $\tau$ of topological type $(m_1,\dots,m_\ell;0)$.
\end{enumerate}
\end{corollary}
\begin{proof}
The disc components of a 
non-orientable recurrent 
train track $\tau$ of topological type $(m_1,\dots,m_\ell;-m)$ 
can be  subdivided
in $4g-4+m-\ell$ steps into trigons by successively adding small
branches. A successive application of Lemma \ref{simpleex} shows
that the resulting train track $\eta$ is maximal and
recurrent. Since 
for every maximal recurrent train track $\eta$ we have
${\rm dim}{\cal V}(\eta)=6g-6+2m$ (see \cite{PH92}), 
the first part of the corollary
follows.

To show the second part, let $\tau$ be an orientable recurrent train track
of type $(m_1,\dots,m_\ell;0)$. Then $m_i$ is even
for all $i$.
Add a branch $b_0$ to $\tau$ which
cuts some complementary component of $\tau$ into a trigon
and a second polygon with an odd number of sides. The resulting
train track $\eta_0$ is not recurrent since a trainpath on $\eta_0$
can only pass through $b_0$ at most once. However, we can add 
to $\eta_0$ another small branch $b_1$ which cuts some complementary
component of $\eta_0$ with at least 4 sides into a trigon and a second 
polygon such that the resulting
train track $\eta$ is non-orientable
and recurrent. The inward pointing tangent of $b_1$
is chosen in such a way that there is a trainpath traveling
both through $b_0$ and $b_1$. The counting measure of any simple 
closed curve which is carried by $\eta$ gives equal weight to the branches
$b_0$ and $b_1$. But this just means that ${\rm dim}{\cal V}(\eta)=
{\rm dim}{\cal V}(\tau)+1$ (see the proof of Lemma \ref{simpleex}
for a detailed argument). 
By the first part of the corollary,
we have ${\rm dim}{\cal V}(\eta)=2g-2+\ell +2$ which completes the proof.
\end{proof}

\begin{definition}\label{deffullyrec} 
A train track $\tau$ of topological type $(m_1,\dots,m_\ell;-m)$
is \emph{fully recurrent} if $\tau$  carries
a large minimal geodesic lamination 
$\nu\in {\cal L\cal L}(m_1,\dots,m_\ell;-m)$.
\end{definition}

Note that by definition, a fully recurrent train track is connected
and fills up $S$. The next lemma gives
some first property of a fully recurrent train track $\tau$. 
For its
proof, recall that 
there is a natural homeomorphism of ${\cal V}(\tau)$ onto the
subspace of ${\cal M\cal L}$ of all measured geodesic laminations
carried by $\tau$.

\begin{lemma}\label{fullyrec}
A fully recurrent
train track $\tau$ of topological type
$(m_1,\dots,m_\ell;-m)$ 
is recurrent.
\end{lemma} 
\begin{proof}
A fully recurrent train track $\tau$ of type
$(m_1,\dots,m_\ell;-m)$ carries a large
geodesic lamination $\nu\in {\cal L\cal L}(m_1,\dots,m_\ell;-m)$.
The carrying map $\nu\to \tau$ induces
a bijection between the complementary
components of $\tau$ and the complementary components of $\nu$.
In particular, a carrying map $\nu\to \tau$ is necessarily
surjective. 
The third paragraph in the proof of Lemma 2.3 of \cite{H09a} now shows
that $\tau$ is recurrent.
\end{proof}

There are two simple ways to modify a fully recurrent
train track $\tau$
to another fully recurrent train track.
Namely, if $b$ is a mixed branch of $\tau$ 
then we can \emph{shift} $\tau$
along $b$ to a new train track $\tau^\prime$. 
This new train track carries $\tau$ and hence it 
is fully recurrent since it carries 
every geodesic lamination
which is carried by $\tau$ \cite{PH92,H09a}.

Similarly, 
if $e$ is a large branch of $\tau$ then we can perform a
right or left \emph{split} of $\tau$ at $e$
as shown in Figure A below.
A (right or left) split $\tau^\prime$ of a 
train track $\tau$ is carried
by $\tau$. 
If $\tau$ is of topological type 
$(m_1,\dots,m_\ell;-m)$, 
if $\nu\in {\cal L\cal L}(m_1,\dots,m_\ell;-m)$
is minimal and is carried by $\tau$ and if $e$ is a large branch
of $\tau$, then there is a unique choice of a right or
left split of $\tau$ at $e$ such that the split track $\eta$ 
carries $\nu$. In particular, $\eta$ is fully recurrent. 
Note however that 
there may be a split of $\tau$ at $e$ such that
the split track is not fully recurrent any more
(see Section 2 of \cite{H09a} for details).
\begin{figure}[ht]
\begin{center}
%\psfrag{e}{$e$}
%\psfrag{b}{$c_3$}
%\psfrag{c}{$c_5$}
\psfrag{Figure A}{Figure A} 
\includegraphics[width=0.8\textwidth]{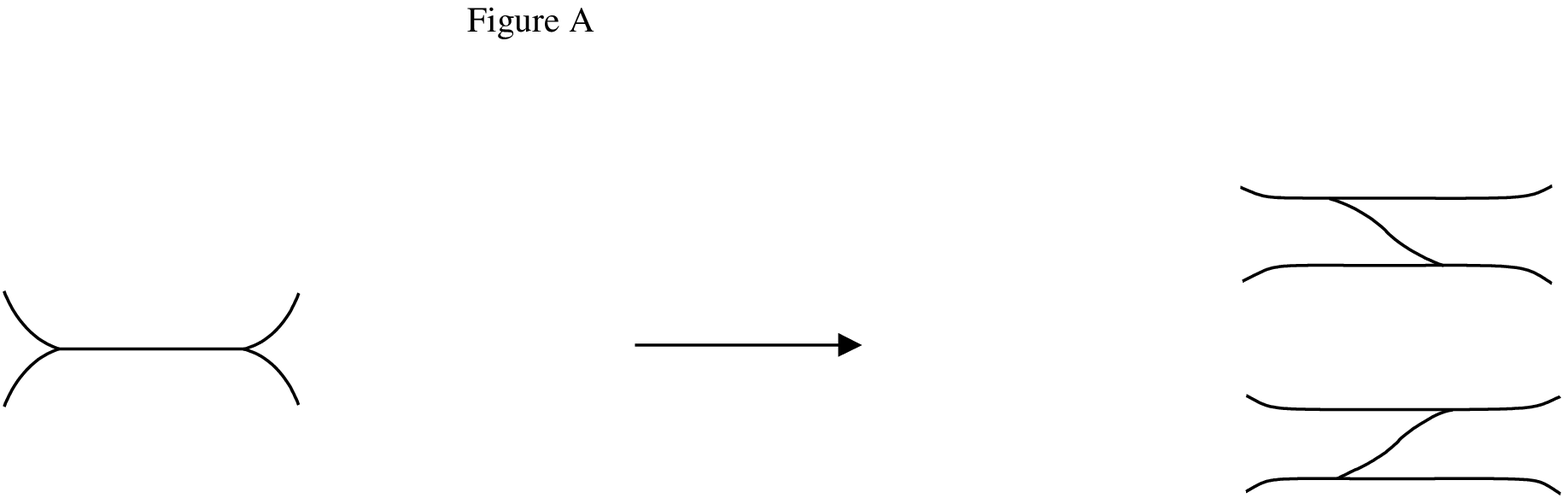}
\end{center}
\end{figure}

The following simple observation is used to identify
fully recurrent train tracks.

\begin{lemma}\label{nosplitor}
\begin{enumerate}
\item 
Let $e$ be a large branch of a fully recurrent non-orientable
train track $\tau$. Then no component of the train track $\sigma$ obtained
from $\tau$ by splitting $\tau$ at $e$ and removing the
diagonal of the split is orientable. 
\item Let $e$ be a large branch of a fully recurrent orientable
train track $\tau$. Then the train track $\sigma$ obtained
from $\tau$ by splitting $\tau$ at $e$ and removing the diagonal
of the split is connected.
\end{enumerate}
\end{lemma}
\begin{proof}
Let $\tau$ be a fully recurrent non-orientable train track
of topological type $(m_1,\dots,m_\ell;-m)$. 
Let $e$ be a large branch of $\tau$ and let 
$v$ be a switch on which
the branch $e$ is incident. Let $\sigma$ be the  
train track obtained from $\tau$ by splitting $\tau$ at $e$ and 
removing the diagonal branch of the split. Then the train tracks
$\tau_1,\tau_2$ obtained from $\tau$ by a right and left split at $e$,
respectively, are simple extensions of $\sigma$.  

If $\sigma$ is connected and orientable then the train tracks 
$\tau_1,\tau_2$ are not recurrent 
since no transverse
measure can give positive weight to the diagonal of the split
(compare the discussion in the proof of Lemma \ref{simpleex}).
However, since $\tau$
is fully recurrent, it can be split
at $e$ to a fully recurrent and hence recurrent train track.
This is a contradiction. 

Now assume that $\sigma$ is disconnected
and contains an orientable connected component $\sigma_1$.
Let $b_i\in \tau_i-\sigma$ be a diagonal of the
split connecting $\tau$ to $\tau_i$ $(i=1,2)$.  
If $\rho_i:[0,m]\to \tau_i$ is a trainpath with
$\rho_i[0,1]=b_i$ and $\rho_i[1,2]\in \sigma_1$ then 
$\rho_i[1,m]\subset
\sigma_1$ and hence once again, $\tau_i$ is not recurrent.
As above, this contradicts the assumption that $\tau$ is fully
recurrent . 
The first part of the corollary is proven. The second
part follows from the same argument since a split of an
orientable train track is orientable.
\end{proof}

\bigskip

{\bf Example:} 1) 
Figure B below shows a non-orientable recurrent train track 
$\tau$ of type $(4;0)$ on a 
closed surface of genus two.
The train track obtained from $\tau$ by a split
at the large branch $e$ and removal of the diagonal of 
the split track is 
orientable and hence $\tau$ is not fully recurrent. This corresponds
to the fact established by Masur and Smillie \cite{MS93} that
every quadratic differential with a single zero and no pole on 
a surface of genus $2$ is the square of a holomorphic one-form
(see Section 4 for more information).
\begin{figure}[ht]
\begin{center}
\psfrag{e}{$e$}
%\psfrag{b}{$c_3$}
%\psfrag{c}{$c_5$}
\psfrag{Figure B}{Figure B} 
\includegraphics[width=0.8\textwidth]{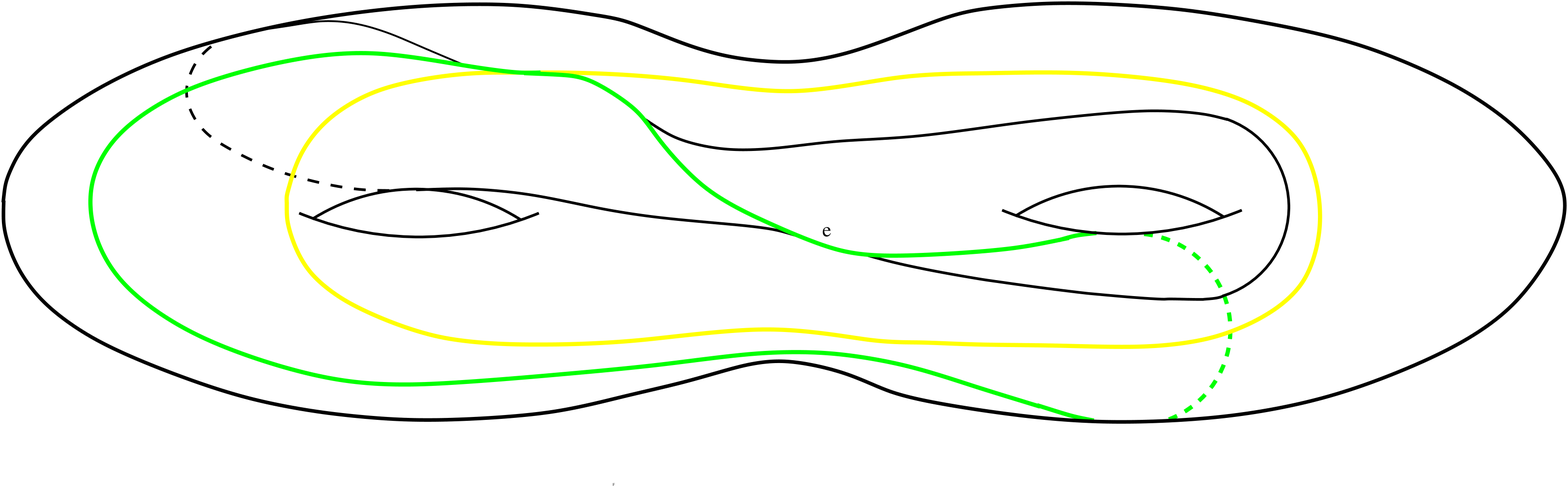}
\end{center}
\end{figure}

2) To construct an orientable recurrent train track
of type $(m_1,\dots,m_\ell;0)$ which is not fully recurrent
let $S_1$ be a surface of genus $g_1\geq 2$ and let 
$\tau_1$ be an orientable fully recurrent train track on $S_1$
with $\ell_1\geq 1$ complementary components.
Choose a complementary component $C_1$ of
$\tau_1$ in $S_1$, remove
from $C_1$ a disc $D_1$ and glue two copies of $S_1-D_1$ along
the boundary of $D_1$ to a surface $S$ of genus $2g_1$. 
The two copies of $\tau_1$
define a recurrent
disconnected oriented train track $\tau$ on $S$ which has an annulus
complementary component $C$.

Choose a branch $b_1$ of $\tau$ 
in the boundary of $C$. There is a corresponding
branch $b_2$ in the second boundary component of $C$. Glue a 
compact subarc of $b_1$ contained in the interior of $b_1$ to a compact
subarc of $b_2$ contained in the interior of $b_2$ so that the
images of the two arcs under the glueing form a large branch $e$ in 
the resulting train track $\eta$. The train track $\eta$ is recurrent
and orientable, and its complementary components are topological discs.
However, by Lemma \ref{nosplitor} it is not fully recurrent.

\bigskip

To each train track $\tau$ 
which fills up $S$ one can
associate a \emph{dual bigon track} $\tau^*$ 
(Section 3.4 of \cite{PH92}).
There is a bijection between
the complementary components of $\tau$ and those
complementary components of $\tau^*$ which are
not \emph{bigons}, i.e. discs with two cusps at the
boundary. This bijection maps
a component $C$ of $\tau$ which is an $n$-gon for some
$n\geq 3$ to an $n$-gon component of 
$\tau^*$ contained in $C$, and it maps a once punctured
monogon $C$ to a once punctured monogon contained in $C$.
If $\tau$ is orientable then the orientation of $S$ and
an orientation of $\tau$ induce an orientation on 
$\tau^*$, i.e. $\tau^*$ is orientable.

Measured geodesic laminations which are carried by $\tau^*$ 
can be described as follows.
A \emph{tangential measure} on a 
train track $\tau$ of type $(m_1,\dots,m_\ell;-m)$ 
assigns to a branch $b$ of $\tau$
a weight $\mu(b)\geq 0$ such that for every
complementary $k$-gon of $\tau$ with consecutive sides
$c_1,\dots,c_k$ and total mass $\mu(c_i)$ (counted
with multiplicities) the following holds true.
\begin{enumerate}
\item $\mu(c_i)\leq \mu(c_{i-1})+
\mu(c_{i+1})$.
\item $\sum_{i=j}^{k+j-1}(-1)^{i-j}\mu(c_i)\geq 0$, $j=1,\dots,k$.
\end{enumerate}
(The complementary once punctured monogons
define no constraint on tangential measures).
The space of
all tangential measures on $\tau$ has the structure
of a convex cone in a finite dimensional real vector space.
By the results from Section 3.4 of \cite{PH92},
every tangential measure on $\tau$ determines a simplex of
measured geodesic laminations which
\emph{hit $\tau$ efficiently}. The supports of these
measured geodesic laminations 
are carried by the bigon track $\tau^*$,
and every measured geodesic
lamination which is carried by $\tau^*$ can be obtained in this way.
The dimension of this simplex equals the number of
complementary components of $\tau$ with an even number of sides.
The train track $\tau$ is called \emph{transversely recurrent}
if it admits a tangential measure which is positive on
every branch.

In general, there are many tangential measures
which correspond to a fixed measured geodesic lamination
$\nu$ which hits $\tau$ efficiently. Namely, let
$s$ be a switch of $\tau$ and let $a,b,c$ be the 
half-branches of $\tau$ incident on $s$ and such that
the half-branch $a$ is large. If $\beta$ is a tangential
measure on $\tau$ which determines the measured
geodesic lamination $\nu$ then it may
be possible to drag the switch $s$ across some of 
the leaves of $\nu$ and modify the tangential
measure $\beta$ on $\tau$ to a tangential measure
$\mu\not=\beta$. Then $\beta-\mu$ is a multiple of 
a vector of the form $\delta_a-\delta_b-\delta_c$ where
$\delta_w$ denotes the function on the
branches of $\tau$ defined by $\delta_w(w)=1$ and
$\delta_w(a)=0$ for $a\not=w$.

\begin{definition}\label{large}
A train track $\tau$ of topological type 
$(m_1,\dots,m_\ell;-m)$ 
is called \emph{fully transversely recurrent}
if its dual bigon track 
$\tau^*$ carries a large minimal geodesic lamination
$\nu\in {\cal L\cal L}(m_1,\dots,m_\ell;-m)$. 
A train track $\tau$ of topological type $(m_1,\dots,m_\ell;-m)$ 
is called \emph{large} if
$\tau$ is fully recurrent and fully transversely recurrent.
A large train track of type $(1,\dots,1;-m)$ is called
\emph{complete}.
\end{definition}

For a large train track $\tau$ let 
${\cal V}^*(\tau)\subset {\cal M\cal L}$ 
be the set of all measured geodesic
laminations whose support is carried by $\tau^*$. Each of 
these measured geodesic
laminations corresponds to a tangential measure on 
$\tau$.  With this identification,
the pairing
\begin{equation}\label{intersectionpairing}
(\nu,\mu)\in {\cal V}(\tau)\times 
{\cal V}^*(\tau)\to \sum_b\nu(b)\mu(b)
\end{equation}
is just the restriction of the intersection form
on measured lamination space
(Section 3.4 of \cite{PH92}).
Moreover, 
${\cal V}^*(\tau)$ is naturally homeomorphic to 
a convex cone in a real vector space. The dimension of this cone
coincides with the dimension of ${\cal V}(\tau)$.

Denote by ${\cal L\cal T}(m_1,\dots,m_\ell;-m)$ the set of all
isotopy classes of large train tracks on $S$ of type
$(m_1,\dots,m_\ell;-m)$.

{\bf Remark:} In \cite{MM99}, Masur and Minsky
define a large train track to be a train track $\tau$ whose
complementary components are topological discs or once
punctured monogons, without the requirement that $\tau$ is
generic, transversely recurrent or recurrent. 
We hope that this inconsistency of terminology does not 
lead to any confusion.

\section{Strata}

As in Section 2, for a closed oriented surface 
$S$ of genus $g\geq 0$ with $m\geq 0$ punctures
let ${\cal Q}^1(S)$ be the
bundle of marked area one holomorphic
quadratic differentials with a simple pole at each puncture
over the Teichm\"uller space ${\cal T}(S)$ 
of marked complex structures
on $S$.
For a complete hyperbolic metric on $S$ of
finite area, an area one quadratic differential 
$q\in {\cal Q}^1(S)$ is 
determined by a pair $(\lambda^+,\lambda^-)$ of 
measured geodesic laminations 
which jointly fill up $S$ and such that
$\iota(\lambda^+,\lambda^-)=1$. 
The \emph{vertical} measured geodesic 
lamination $\lambda^+$ for $q$
corresponds to the equivalence class of the vertical measured
foliation of $q$. 
The \emph{horizontal} measured geodesic lamination
$\lambda^-$ for $q$ corresponds to the equivalence
class of the horizontal measured foliation of $q$.

A tuple $(m_1,\dots,m_\ell)$ of positive integers 
$1\leq m_1\leq \dots \leq m_\ell$ with $\sum_im_i=4g-4+m$
defines a \emph{stratum} ${\cal Q}^1(m_1,\dots,m_\ell;-m)$ 
in ${\cal Q}^1(S)$. This stratum consists 
of all marked area one
quadratic differentials with $m$ simple poles and
$\ell$ zeros of 
order $m_1,\dots,m_\ell$ which are not squares of holomorphic
one-forms. 
The stratum is a real hypersurface
in a complex manifold of dimension 
\begin{equation}\label{h}
h=2g-2+m+\ell.\end{equation}
The closure in ${\cal Q}^1(S)$ of 
a stratum is a union of components of strata. 
Strata are invariant
under the action of the mapping class group
${\rm Mod}(S)$ of $S$ 
and hence
they project to strata in the moduli space
${\cal Q}(S)={\cal Q}^1(S)/{\rm Mod}(S)$ 
of quadratic differentials on $S$ with a simple
pole at each puncture. We denote the
projection of the stratum ${\cal Q}^1(m_1,\dots,m_\ell;-m)$ by
${\cal Q}(m_1,\dots,m_\ell;-m)$. The strata in moduli
space need not be connected, but their connected
components have been identified by 
Lanneau \cite{L08}. A stratum in ${\cal Q}(S)$ has
at most two connected components.

Similarly, if $m=0$ then we 
let ${\cal H}^1(S)$ be the bundle of marked area one
holomorphic one-forms over Teichm\"uller space
${\cal T}(S)$ of $S$. 
For a tuple $k_1\leq \dots \leq k_\ell$
of positive integers with $\sum_ik_i=2g-2$, the stratum 
${\cal H}^1(k_1,\dots,k_\ell)$ of marked area one holomorphic one-forms
on $S$ with $\ell$ zeros of order $k_i$ $(i=1,\dots,\ell)$
is a real hypersurface in a complex manifold of dimension
\begin{equation}\label{h2} h=2g-1+\ell.\end{equation} 
It projects to a stratum 
${\cal H}(k_1,\dots,k_\ell)$ in the moduli space
${\cal H}(S)$ of area one holomorphic one-forms on $S$. 
Strata of holomorphic one-forms in moduli space
need not be connected, 
but the number
of connected components of a stratum is at most three
\cite{KZ03}.

Recall from Section 2 the definition
of the strong stable, the stable, the unstable and the strong unstable foliation
$W^{ss},W^s,W^u,W^{su}$ of ${\cal Q}^1(S)$.
Let $\tilde{\cal Q}$ be a component of
a stratum ${\cal Q}^1(m_1,\dots,m_\ell;-m)$ of marked quadratic differentials
or of a stratum
${\cal H}^1(m_1/2,\dots,m_\ell/2)$ of marked abelian differentials. 
Using period coordinates, one sees that 
every $q\in \tilde {\cal Q}$ has a 
connected neighborhood
$U$ in $\tilde {\cal Q}$ 
with the following properties \cite{V90}. For $u\in U$ let 
$[u^v]$ (or $[u^h]$) be the vertical (or the horizontal)
projective measured geodesic lamination of $u$. 
Then $\{[u^v]\mid u\in U\}$ is homeomorphic to an open ball
in $\mathbb{R}^{h-1}$ (where $h>0$ is as in 
equation (\ref{h},\ref{h2})). Moreover, for $q\in U$ the set
\[\{u\in U\mid [u^v]=[q^v]\}=W^s_{\tilde{\cal Q},{\rm loc}}(q)\subset W^s(q)\] 
is a smooth connected local submanifold of $U$ 
of (real) dimension $h$ which
is called the \emph{local stable manifold} of $q$ in $\tilde {\cal Q}$ 
(see \cite{V90}).
Similarly we define the \emph{local unstable manifold}
$W^u_{\tilde {\cal Q},{\rm loc}}(q)$ of $q$ in $\tilde {\cal Q}$. 
If two such local stable
(or unstable) manifolds intersect then their union is again 
a local stable (or unstable) 
manifold.
The maximal connected set containing
$q$ which is a union of
 intersecting local stable (or unstable) manifolds
is the \emph{stable manifold} $W^s_{\tilde{\cal Q}} (q)$
(or the \emph{unstable manifold} $W^u_{\tilde {\cal Q}}(q)$) of $q$
in $\tilde {\cal Q}$. Note that $W^i_{\tilde {\cal Q}}(q)\subset
W^i(q)$ $(i=s,u)$.
A stable (or unstable) manifold is invariant under the action of the
Teichm\"uller flow $\Phi^t$.

{\bf Remark:} There may be a component $\tilde {\cal Q}$
of a stratum and some $\tilde q\in \tilde {\cal Q}$ such that
$W^s(\tilde q)\cap \tilde {\cal Q}$ has infinitely many components.

The (strong) stable and (strong) unstable manifolds define
smooth
foliations $W_{\tilde {\cal Q}}^s,W_{\tilde {\cal Q}}^u$ 
of $\tilde{\cal Q}$ which 
are called the \emph{stable} and \emph{unstable}
foliations of $\tilde {\cal Q}$, respectively. 
Define the \emph{strong stable} foliation 
$W^{ss}_{\tilde {\cal Q}}$
(or the \emph{strong unstable} foliation 
$W^{su}_{\tilde {\cal Q}}$)  of $\tilde {\cal Q}$ 
by requiring that locally the leaf $W^{ss}_{\tilde {\cal Q}}(q)$ 
(or $W^{su}_{\tilde {\cal Q}}(q)$)  
through $q$ is the subset of $W^s_{\tilde {\cal Q}}(q)$ 
(or of $W^u_{\tilde {\cal Q}}(q)$) of all marked
quadratic differentials 
whose vertical (or horizontal) measured geodesic lamination
equals the vertical (or horizontal) measured geodesic lamination of $q$.
The strong stable foliation of $\tilde {\cal Q}$ is transverse to 
the unstable foliation of $\tilde {\cal Q}$.

The foliations $W^i_{\tilde {\cal Q}}$ $(i=ss,s,su,u)$ are invariant under 
the action of the stabilizer
${\rm Stab}(\tilde {\cal Q})$ of $\tilde {\cal Q}$ in 
${\rm Mod}(S)$, and they  
project to $\Phi^t$-invariant singular foliations $W^i_{\cal Q}$ of 
${\cal Q}=\tilde {\cal Q}/{\rm Stab}(\tilde {\cal Q})$.

\subsection{Orbifold coordinates}

In this 
technical subsection we 
describe for every component
${\cal Q}$ of a stratum in the moduli space
of quadratic differentials  and for every point
$q\in {\cal Q}$ 
a basis of neighborhoods of $q$ in ${\cal Q}$ 
with local product structures. The material is well known
to the experts but a bit difficult to find in the literature.
In the course of the discussion we introduce some
notations which will be used throughout.

For $\tilde q\in {\cal Q}^1(S)$ 
and $z\in W^{s}(\tilde q)$ there is a neighborhood
$V$ of $\tilde q$ in $W^{su}(\tilde q)$ 
and there is a homeomorphism 
\begin{equation}\label{zeta}
\zeta_z:V\to \zeta_z(V)\subset W^{su}(z)\end{equation}
with $\zeta_z(\tilde q)=z$ which is
determined by the requirement
that $\zeta_z(u)\in W^s(u)$. We call $\zeta_z$ a 
\emph{holonomy map} for the strong unstable foliation along the
stable foliation.

Similarly, for $\tilde q\in {\cal Q}^1(S)$ and 
$z\in W^{u}(\tilde q)$ there is a neighborhood $Y$ of $\tilde q$
in $W^{ss}(\tilde q)$ and there is a homeomorphism
\begin{equation}\label{theta}
\theta_z:Y\to \theta_z(Y)\subset W^{ss}(z)\end{equation}
with $\theta_z(\tilde q)=z$ which is
determined by the requirement
that $\theta_z(u)\in W^u(u)$. We call $\theta_z$ a 
\emph{holonomy map} for the strong stable foliation along the
unstable foliation. 
The holonomy maps 
are equivariant under the action of the mapping class group
and hence they project to locally defined
holonomy maps in ${\cal Q}(S)$ which are denoted by the same symbols.

Recall from Section 2 the definition of the intrinsic 
path-metrics $d^i$ on the leaves of the foliation $W^i$
$(i=s,u)$. These path metrics
are invariant under the action of the mapping class group
and hence they project to path metrics on the leaves of
$W^i$ in ${\cal Q}(S)$ which we denote by the same symbols.
For $q\in {\cal Q}(S),z\in W^i(q)$ and any
preimage $\tilde q$ of $q$ in ${\cal Q}^1(S)$, the distance
$d^i(q,z)$ is the shortest length of a path in 
$W^i(\tilde q)$ connecting $\tilde q$ to a preimage of $z$.
Let moreover $d^{ss},d^{su}$ be the 
restrictions of $d^s,d^u$ to distances
on the leaves of the strong stable and strong unstable foliation
of ${\cal Q}^1(S)$ and ${\cal Q}(S)$.

Let  \[\Pi:{\cal Q}^1(S)\to {\cal Q}(S)\]
be the canonical projection.
For $q\in {\cal Q}(S)$ and $r>0$ let 
\[B^i(q,r)\] 
be the closed ball of radius $r$ about $q$ in 
$W^i(q)$ $(i=ss,su,s,u)$ with respect to the metric $d^i$. 
Call such a ball $B^{i}(q,r)$ a \emph{metric orbifold ball} centered at $q$
if there is a lift $\tilde q\in {\cal Q}^1(S)$ of $q$
with the following properties.
\begin{enumerate}
\item The ball $B^i(\tilde q,r)\subset
(W^{i}(\tilde q),d^i)$ about $\tilde q$ of the same radius is contractible and
precisely invariant under the stabilizer 
${\rm Stab}(\tilde q)$ of 
$\tilde q$ in ${\rm Mod}(S)$.
\item $B^i(q,r)= B^i(\tilde q,r)/{\rm Stab}(\tilde q)$ which
means that the restriction 
of the map $\Pi$ to $B^i(\tilde q,r)$ factors through 
a homeomorphism $B^i(\tilde q,r)/{\rm Stab}(\tilde q)\to  
B^i(q,r)$.
\end{enumerate}
We also say that $B^i(q,r)$ is an \emph{orbifold quotient}
of $B^i(\tilde q,r)$. Note that every
metric orbifold ball $B^i(q,r)\subset W^i(q)$ is contractible.
There is also an obvious notion of an orbifold ball which
is not necessarily metric.

For every point $q\in {\cal Q}(S)$ there is a number
\[a(q)>0\] 
such that the balls $B^i(q,a(q))$ 
are metric orbifold balls $(i=ss,su)$ 
and that for any preimage $\tilde q$ of $q$ in ${\cal Q}^1(S)$ 
and any $z\in B^{ss}(\tilde q,a(q))$ 
(or $z\in B^{su}(\tilde q,a(q))$)
the holonomy map $\zeta_z$ (or $\theta_z$)
is defined on $B^{su}(\tilde q,a(q))$
(or on $B^{ss}(\tilde q,a(q))$).

Now let \[W_1\subset B^{ss}(q,a(q)),
W_2\subset B^{su}(q,a(q))\]
be Borel sets
and let $\tilde W_1\subset B^{ss}(\tilde q,a(q))$,
$\tilde W_2\subset B^{su}(\tilde q,a(q))$ 
be the preimages of $W_1,W_2$ 
in $B^{ss}(\tilde q,a(q)),B^{su}(\tilde q,a(q))$.
Then $\tilde W_1,\tilde W_2$ are precisely invariant under
${\rm Stab}(\tilde q)$.
Define 
\begin{equation}V(\tilde W_1,\tilde W_2) =
\cup_{z\in \tilde W_1}\zeta_z\tilde W_2
\text{ and }
V(W_1,W_2)  =\Pi V(\tilde W_1,\tilde W_2).\notag
\end{equation} 
Note that the map $\xi:\tilde W_1\times \tilde W_2\to 
V(\tilde W_1,\tilde W_2)$ defined by 
$\xi(z,u)=\zeta_z(u)$ is a homeomorphism, and 
$V(W_1,W_2)$ is 
homeomorphic to the quotient of $\tilde W_1\times \tilde W_2$
under the diagonal action of ${\rm Stab}(\tilde q)$.
In particular, if $W_1,W_2$ are
connected then $V(W_1,W_2)$ is connected as well.
Similarly, define 
\[Y(\tilde W_1,\tilde W_2)=\cup_{z\in \tilde W_2}
\theta_z \tilde W_1\text{ and }Y(W_1,W_2)=\Pi Y(\tilde W_1,\tilde W_2).\]
Then there is a continuous function 
\begin{equation}\label{sigma}
\sigma:V(B^{ss}(q,a(q)),B^{su}(q,a(q)))\to \mathbb{R}\end{equation}
which vanishes on $B^{ss}(q,a(q))\cup B^{su}(q,a(q))$ and such that
\[Y(W_1,W_2)=\{\Phi^{\sigma(z)}z\mid z\in V(W_1,W_2)\}.\]
In particular, for every number $\kappa>0$ there is a number
$r(\kappa)>0$ such that the restriction
of the function $\sigma$ 
to $V(B^{ss}(\tilde q,r(\kappa)),B^{su}(\tilde q,r(\kappa)))$
assumes values in $[-\kappa,\kappa]$.

For $t_0>0$ define
\begin{align}\label{localproduct}
V(\tilde W_1,\tilde W_2,t_0)= &\cup_{-t_0\leq s\leq t_0}
\Phi^sV(\tilde W_1,\tilde W_2)\\
\text{ and }V(W_1,W_2,t_0)=& 
\Pi V(\tilde W_1,\tilde W_2,t_0).
\notag
\end{align}
Then for sufficiently small $t_0$, say for all $t_0\leq t(q)$, 
the following properties are satisfied.

\begin{enumerate}
\item[a)] 
$V(W_1,W_2,t_0)$ 
is homeomorphic to 
$(\tilde W_1\times \tilde W_2)/{\rm Stab}(\tilde q)\times [-t_0,t_0]$.
\item[b)]
Every connected component of the intersection of 
an orbit of $\Phi^t$ with $V(W_1,W_2,t_0)$ is an arc of length
$2t_0$. 
\end{enumerate}

We call a set $V(W_1,W_2,t_0)$ as in (\ref{localproduct})
which satisfies the assumptions a),b)
a set with a \emph{local product structure}. Note that
every point $q\in {\cal Q}(S)$ has a neighborhood
in ${\cal Q}(S)$ with a local product structure, e.g. 
the set $V(B^{ss}(q,r),B^{su}(q,r),t)$ for
$r\in(0,a(q))$ and $t\in (0,t(q))$. Moreover, the
neighborhoods of $q$ with a local product structure form a 
basis of neighborhoods.

The above discussion can be applied to strata 
as follows.   

A connected component ${\cal Q}$ of 
a stratum ${\cal Q}(m_1,\dots,m_\ell;-m)$ 
or of a stratum
${\cal H}(m_1/2,\dots,m_\ell/2)$ is locally
closed in ${\cal Q}(S)$ (here we identify an abelian differential
with its square).
This means that for every $q\in {\cal Q}$ 
there exists an open neighborhood $V$ of $q$ in ${\cal Q}(S)$ such that
$V\cap {\cal Q}$ is a closed subset of $V$.

%Moreover, if $\hat q\in 
%\hat{\cal Q}-\tilde {\cal Q}$ then for every
%neighborhood $U$ of $\hat q$ in $\hat{\cal Q}$ 
%there is a quadratic
%differential $u\in U$ whose horizontal
%measured geodesic lamination coincides with the
%horizontal measured geodesic lamination of $\hat q$, and
%there is a quadratic differential $z\in U$ whose
%vertical measured geodesic lamination coincides with
%the vertical measured geodesic lamination of $\hat q$.  

Using period coordinates \cite{V90}, one obtains that 
for every point $q\in {\cal Q}$ 
there is a number $a_{\cal Q}(q)\leq a(q)$ 
and a number $t_{\cal Q}(q)\leq t(q)$ 
with the following property.
For $r\leq a_{\cal Q}(q)$ let 
\[B^{ss}_{\cal Q}(q,r),B^{su}_{\cal Q}(q,r)\] be the component 
containing $q$ of the intersection 
$B^{ss}(q,r)\cap {\cal Q},B^{su}(q,r)\cap {\cal Q}$
(note that the intersection
$B^{ss}(q,r)\cap {\cal Q}$ may not be closed and may have infinitely
many components).
Then $V(B^{ss}_{\cal Q}(q,r),B^{su}_{\cal Q}(q,r),t_{\cal Q}(q))$ is 
a neighborhood of $q$ in ${\cal Q}$ \cite{V90}.
We say that  this neighborhood has a \emph{local product structure}.

We say that a Borel
set $Y\subset {\cal Q}$ 
has a \emph{local product structure} if 
there is some $q\in Y$ and if there are Borel sets
\[W_1\subset B^{ss}_{\cal Q}(q,a(q)),
W_2\subset B^{su}_{\cal Q}(q,a(q))\]
and a number $t_0<t(q)$
such that $Y=V(W_1,W_2,t_0)$.

The $\Phi^t$-invariant Borel probability
measure $\lambda$ on ${\cal Q}$ 
in the Lebesgue measure class admits a natural family
of conditional measures $\lambda^{ss},\lambda^{su}$ on 
strong stable and strong unstable manifolds. 
The conditional
measures $\lambda^i$ are well defined up to a universal constant,
and they 
transform under the 
Teichm\"uller geodesic flow $\Phi^t$ via
\[d\lambda^{ss}\circ \Phi^t=e^{-ht}d\lambda^{ss}\text{ and }
d\lambda^{su}\circ \Phi^t=e^{ht}d\lambda^{su}.\]
Let ${\cal F}:{\cal Q}(S)\to {\cal Q}(S)$ be the 
flip $q\to {\cal F}(q)=-q$ and let 
$dt$ be the Lebesgue measure on the flow
lines of the Teichm\"uller flow.
The conditional measures
$\lambda^{ss},\lambda^{su}$ are uniquely determined
by the additional requirements that
${\cal F}_*\lambda^{su}=\lambda^{ss}$
and that 
with respect to a local product structure, $\lambda$ can be
written in the form 
\[d\lambda=d\lambda^{ss}\times 
d\lambda^{su}\times dt.\] 
 The measures $\lambda^u$ on unstable manifolds
defined by 
$d\lambda^u=d\lambda^{su}\times dt$ are invariant 
under holonomy along strong 
stable manifolds.

To summarize, we obtain the following.
The natural homeomorphism
\begin{align}
\Psi: B^{ss}_{\cal Q}(q,a_{\cal Q}(q))
\times B^{su}_{\cal Q}(q,a_{\cal Q}(q))\times 
[-t_{\cal Q}(q),t_{\cal Q}(q)]\notag\\
\to V(B^{ss}_{\cal Q}(q,a_{\cal Q}(q)),
B^{su}_{\cal Q}(q,a_{\cal Q}(q)),t_{\cal Q}(q))=V\notag
\end{align} 
maps the measure $\lambda_0$
on $V$ defined by
$d\lambda_0=\Psi_*d\lambda^{ss}\times d\lambda^{su}\times dt$
to a measure of the form $e^\phi\lambda$ 
where $\phi$ is a
continuous function on $V$ which vanishes on 
$\cup_{t\in [-t_{\cal Q}(q),t_{\cal Q}(q)]}
\Phi^t B^{ss}_{\cal Q}(q,a_{\cal Q}(q))$
(see \cite{V86}).

\subsection{Train track coordinates}

The goal of this subsection is to relate components of strata
in ${\cal Q}(S)$ to large train tracks.
This will be used to define product coordinates
near points in the boundary 
of a stratum. Note that the natural product coordinates
on strata are period coordinates. For a point $q$ 
in the boundary 
of a stratum, some of the relative periods vanish and there is
no canonical  choice of a relative cycle near $q$ which can 
be used for period coordinates in a neighborhood of $q$.

We chose to construct product coordinates near boundary points of  
a stratum using train tracks even
though similar coordinates can be obtained using the usual period coordinate
construction.  These
train track coordinates will be used in other contexts as well.

We continue to use the assumptions and notations from Section 2 and Section 3.
For a large train track $\tau\in {\cal L\cal T}(m_1,\dots,m_\ell;-m)$ 
let 
\[{\cal V}_0(\tau)\subset {\cal V}(\tau)\] be the set of all measured
geodesic laminations $\nu\in {\cal M\cal L}$ 
whose support is carried by $\tau$
and such that the total weight of the transverse measure
on $\tau$ defined by $\nu$ equals one. 
Let 
\[{\cal Q}(\tau)\subset {\cal Q}^1(S)\] be 
the set of all area one marked quadratic differentials 
whose vertical measured geodesic lamination 
is contained in ${\cal V}_0(\tau)$ 
and whose horizontal 
measured geodesic lamination is carried by the dual
bigon track $\tau^*$ of $\tau$.
By definition of a large train track, 
we have ${\cal Q}(\tau)\not=\emptyset$.
The next proposition relates 
${\cal Q}(\tau)$ to components of strata. 

\begin{proposition}\label{structure}
\begin{enumerate}
\item 
For every large non-orientable train track\\ 
$\tau\in {\cal L\cal T}(m_1,\dots,m_\ell;-m)$ 
there is a component $\tilde{\cal Q}$ 
of the stratum\\ 
${\cal Q}^1(m_1,\dots,m_\ell;-m)$ such that
for every $\delta >0$ the set 
$\{\Phi^tq\mid q\in {\cal Q}(\tau),
t\in [-\delta,\delta]\}$ is the closure 
in ${\cal Q}^1(S)$  of 
an open subset of $\tilde{\cal Q}$.
\item For every large orientable train track 
$\tau\in {\cal L\cal T}(m_1,\dots,m_\ell;0)$ 
there is a component $\tilde{\cal Q}$ of the stratum
${\cal H}^1(m_1/2,\dots,m_\ell/2)$
such that 
for every $\delta >0$ the set 
$\{\Phi^tq\mid q\in {\cal Q}(\tau),
t\in [-\delta,\delta]\}$ is the closure 
in ${\cal H}^1(S)$ of 
an open subset of $\tilde{\cal Q}$.
\end{enumerate}
\end{proposition}
\begin{proof} By \cite{L83}, the support $\xi$ of the 
vertical measured geodesic lamination of 
a marked quadratic differential $z\in {\cal Q}^1(S)$
can be obtained from the vertical foliation of $z$ by cutting $S$ open
along each vertical separatrix and straightening the remaining 
leaves with respect to the hyperbolic structure $Pz\in {\cal T}(S)$.
In particular, up to homotopy a vertical saddle connection $s$ of $z$  is
contained in the interior of a complementary component $C$ of $\xi$
which is uniquely determined by $s$.

 Let $\tau\in {\cal L\cal T}(m_1,\dots,m_\ell;-m)$. Assume first that 
$\tau$ is non-orientable. Let
$\mu\in {\cal V}_0(\tau)$ be such that 
the support of $\mu$ is contained in 
${\cal L\cal L}(m_1,\dots,m_\ell;-m)$
and let  $\nu\in {\cal V}^*(\tau)$. 
Then $\mu$ is non-orientable
since otherwise $\tau$ inherits an orientation from $\mu$.
The measured geodesic laminations
$\mu,\nu$ jointly fill up $S$ (since the support of 
$\nu$ is different from the support of $\mu$ and the
support of $\mu$ fills up $S$) and hence 
if $\nu$ is normalized in such a way that
$\iota(\mu,\nu)=1$ then the pair $(\mu,\nu)$ defines
a point $q\in {\cal Q}(\tau)$. Our first goal is to show 
that $q\in {\cal Q}^1(m_1,\dots,m_\ell;-m)$.

The support of the geodesic lamination $\mu$ is contained
in ${\cal L\cal L}(m_1,\dots,m_\ell;-m)$ and therefore
the orders of the zeros of the quadratic differential $q$ are obtained from
the orders $m_1,\dots,m_\ell$ by subdivision.
There
is a non-trivial subdivision, say of the form $m_i=\sum_sk_s$, if and only if
there is at least one vertical saddle connection for $q$.

Assume to the contrary that there is a vertical saddle connection $s$ for $q$.
Let $\tilde q$ be the lift of $q$ to 
a quadratic differential on the universal covering 
${\bf H}^2$ of $S$ and let $\tilde s\subset {\bf H}^2$ be 
a preimage of $s$.
Let $\tilde\mu\subset {\bf H}^2$
be the preimage of $\mu$. As discussed in the first
paragraph of this proof,
the saddle connection $\tilde s$ is contained in a complementary
component $\tilde C$ of the support of $\tilde \mu$. 
This component is an ideal polygon with
finitely many sides. 

A biinfinite geodesic line for the singular
euclidean metric defined by $\tilde q$ is a quasi-geodesic
in the hyperbolic plane ${\bf H}^2$
and hence it has well defined endpoints in the ideal boundary $\partial {\bf H}^2$ of
${\bf H}^2$. There are two vertical geodesic lines
$\alpha_0,\beta_0$ for $\tilde q$ which contain
the saddle connection $\tilde s$ as a subarc 
and which are contained in a bounded neighborhood
of a side $\alpha,\beta$ of $\tilde C$. For a fixed orientation of $\tilde s$, the 
geodesics $\alpha_0,\beta_0$ 
 are determined by the requirement that their orientation coincides
with the given orientation of $\tilde s$ and that moreover at every singular point $x$,
the angle at $x$ to the left of $\alpha_0$ (or to the right of $\beta_0$) 
for the orientation of the geodesic and the orientation
of ${\bf H}^2$ equals $\pi$. 

The ideal boundary of the closed half-plane of ${\bf H}^2$ which 
is bounded by 
$\alpha$ (or $\beta$) and which is disjoint from the
interior of $\tilde C$ is a compact
subarc $a$ (or $b$) of $\partial {\bf H}^2$. The arcs $a,b$ are disjoint (or, equivalently,
the sides $\alpha,\beta$ of $\tilde C$ are not adjacent).
A horizontal geodesic line for $\tilde q$ which intersects the interior
of the saddle connection $\tilde s$ is a quasi-geodesic in ${\bf H}^2$ 
with one endpoint in the interior of the arc $a$ and the second
endpoint  in the interior of the arc $b$.

Now a carrying map $F:S\to S$ for $\mu$ with 
$F(\mu)\subset \tau$ maps the support of $\mu$ onto $\tau$ 
and hence it induces a bijection
between the complementary
components of the support of $\mu$ and the complementary
components of $\tau$. 
In particular, the projections of the geodesics $\alpha,\beta$ to $S$ 
determine two opposite sides 
of the complementary component $C_\tau$  of $\tau$ corresponding to the 
projection of $\tilde C$ to $S$.

On the other hand, by construction of the dual bigon 
track $\tau^*$ of $\tau$ (see \cite{PH92}, if $\rho:(-\infty,\infty)\to \tau^*$ is 
any trainpath which intersects 
the complementary component $C_\tau$ 
of $\tau$ then every
component of $\rho(-\infty,\infty)\cap C_\tau$ is a compact arc with
endpoints on adjacent sides of $C_\tau$. 
In particular, a lift to ${\bf H}^2$ 
of such a trainpath is a quasi-geodesic in ${\bf H}^2$ 
whose endpoints meet at most one of the two arcs $a,b\subset \partial {\bf H}^2$.
Since the support of 
the horizontal measured geodesic lamination $\nu$ of $q$  is 
carried by $\tau^*$ by assumption, 
every leaf of the support of $\nu$ corresponds to a biinfinite 
trainpath on $\tau^*$ and hence a lift to ${\bf H}^2$ of such a leaf
does not connect the arcs $a,b\subset \partial{\bf H}^2$. 
This contradicts the
assumption that $q$ has a vertical saddle connection and hence 
we indeed have $q\in {\cal Q}^1(m_1,\dots,m_\ell;-m)$.

Let ${\cal P}(\mu)\subset {\cal P\cal M\cal L}$ be the 
open set of all projective measured
geodesic laminations whose support is distinct from the support of $\mu$. 
Then
the assignment $\psi$ which associates to a projective measured
geodesic lamination $[\nu]\in {\cal P}(\mu)$ 
the area one quadratic differential $q(\mu,[\nu])$ with vertical 
measured geodesic lamination $\mu$ and horizontal
projective measured geodesic lamination $[\nu]$ is a homeomorphism
of ${\cal P}(\mu)$ onto a strong stable manifold in ${\cal Q}^1(S)$.

The  projectivization $P{\cal V}^*(\tau)$ of 
${\cal V}^*(\tau)$ is homeomorphic to a ball in a real
vector space of dimension $h-1$, and this is just
the dimension of a strong stable manifold
in a component of ${\cal Q}^1(m_1,\dots,m_\ell;-m)$.
Therefore by the above discussion
and invariance of domain, there is a component
$\tilde {\cal Q}$ of the stratum
${\cal Q}^1(m_1,\dots,m_\ell;-m)$ such that 
the restriction of the map $\psi$ 
to $P{\cal V}^*(\tau)$ is a homeomorphism 
of $P{\cal V}^*(\tau)$ onto 
the closure of an open subset of a
strong stable manifold $W^{ss}_{\tilde {\cal Q}}(q)\subset \tilde{\cal Q}$.

Similarly, 
if $q\in {\cal Q}(\tau)$
is defined by $\mu\in {\cal V}_0(\tau),\nu\in {\cal V}^*(\tau)$ and if
the support of $\nu$ is contained in 
${\cal L\cal L}(m_1,\dots,m_\ell;-m)$ then
$q\in {\cal Q}^1(m_1,\dots,m_\ell;-m)$ by the 
above argument.
Moreover, for every
$[\mu]\in P{\cal V}(\tau)$ the pair $([\mu],\nu)$ 
defines a quadratic differential which is contained in a
strong unstable manifold $W^{su}_{\tilde {\cal Q}}(q)$ 
of a component $\tilde {\cal Q}$ of 
the stratum ${\cal Q}^1(m_1,\dots,m_\ell;-m)$, and the set of
these quadratic differentials equals the closure of an open subset
of $W^{su}_{\tilde {\cal Q}}(q)$.

The set of 
quadratic differentials $q$ with the property that the 
support of the vertical (or of the horizontal) 
measured geodesic 
lamination of $q$ is minimal and of type $(m_1,\dots,m_\ell;-m)$
is dense and of full Lebesgue measure in
${\cal Q}^1(m_1,\dots,m_\ell;-m)$
\cite{M82,V86}. Moreover, this set is saturated for the 
stable (or for the unstable) foliation.
Thus by the above discussion, the
set of all measured geodesic laminations which are
carried by $\tau$ (or $\tau^*$) 
and whose support is minimal of type 
$(m_1,\dots,m_\ell;-m)$ 
is dense in ${\cal V}(\tau)$ (or in ${\cal V}^*(\tau)$). 
As a consequence, the set of all pairs
$(\mu,\nu)\in {\cal V}(\tau)\times
{\cal V}^*(\tau)$ with $\iota(\mu,\nu)=1$ 
which correspond to a quadratic
differential $q\in {\cal Q}^1(m_1,\dots,m_\ell;-m)$ is 
dense in the set of all pairs 
$(\mu,\nu)\in 
{\cal V}(\tau)\times {\cal V}^*(\tau)$ with
$\iota(\mu,\nu)=1$. Thus
the set ${\cal Q}(\tau)$ is contained in the closure of a
component $\tilde {\cal Q}$ of the 
stratum ${\cal Q}^1(m_1,\dots,m_\ell;-m)$. 
Moreover, by reasons of dimension,
$\{\Phi^tq\mid q\in {\cal Q}(\tau),t\in [-\delta,\delta]\}$ contains
an open subset of this component. This shows the first part
of the proposition.

Now
if $\tau\in {\cal L\cal T}(m_1,\dots,m_\ell;-m)$ 
is orientable and if $\mu$ is a
geodesic lamination which 
is carried by $\tau$, then $\mu$ inherits 
an orientation from an 
orientation of $\tau$. The orientation of $\tau$
together with the orientation of $S$ determines an orientation
of the dual bigon track $\tau^*$ (see \cite{PH92}, 
and these two orientations
determine the orientation of $S$. This implies that
any geodesic lamination carried by $\tau^*$ admits
an orientation, and if $(\mu,\nu)$ jointly
fill up $S$ and if $\mu$ is carried by $\tau$,
$\nu$ is carried by $\tau^*$ then the orienations
of $\mu,\nu$ determine the orientation of $S$. As a consequence,
the singular euclidean
metric on $S$ defined by the quadratic differential
$q$ of $(\mu,\nu)$ is the square of a holomorphic
one-form. The proposition follows.
\end{proof}

If $\tilde {\cal Q}$ is a component of a stratum 
${\cal Q}^1(m_1,\dots,m_\ell;-m)$ and if the large train
track 
$\tau\in {\cal L\cal T}(m_1,\dots,m_\ell;-m)$ 
is such that ${\cal Q}(\tau)\cap \tilde {\cal Q}\not=
\emptyset$ then we say that $\tau$ \emph{belongs} to 
$\tilde {\cal Q}$, and we write $\tau\in {\cal L\cal T}(\tilde {\cal Q})$.
The next proposition is a converse
to Proposition \ref{structure} and  shows that train tracks can be used to define
coordinates on strata.

\begin{proposition}\label{all}
\begin{enumerate}
\item For every
$q\in {\cal Q}^1(m_1,\dots,m_\ell;-m)$ 
there is a large non-orientable train track
$\tau\in {\cal L\cal T}(m_1,\dots,m_\ell;-m)$ 
and a number
$t\in \mathbb{R}$ 
so that $\Phi^tq$ is an interior point of ${\cal Q}(\tau)$.
\item For every $q\in {\cal H}^1(k_1,\dots,k_s)$
there is a large orientable train track 
$\tau\in {\cal L\cal T}(2k_1,\dots,2k_s;0)$ 
and a number
$t\in \mathbb{R}$ so that $\Phi^tq$ is an interior
point of ${\cal Q}(\tau)$.
\end{enumerate}
\end{proposition}
\begin{proof}
Fix a complete hyperbolic metric on $S$ of finite
volume. Define the \emph{straightening} of a train track
$\tau$ to be the immersed graph in $S$ whose vertices
are the switches of $\tau$ and whose edges are
the geodesic arcs which are homotopic to the
branches of $\tau$ with fixed endpoints. 

The hyperbolic metric induces a distance function on the 
projectivized tangent bundle of $S$.
As in Section 3 of \cite{H09a}, we say
that for some $\epsilon >0$ 
a train track $\tau$ \emph{$\epsilon$-follows} a
geodesic lamination $\mu$ if
the tangent lines of the straightening of $\tau$
are contained in the $\epsilon$-neighborhood of the 
tangent lines of $\mu$ in the projectivized tangent bundle of $S$ 
and if moreover the
straightening of any trainpath on $\tau$ is a piecewise
geodesic whose exterior angles at the breakpoints are
not bigger than $\epsilon$.
By Lemma 3.2 of \cite{H09a}, for every geodesic lamination 
$\mu$ and 
every $\epsilon>0$ there is a transversely recurrent
train track which carries $\mu$ and $\epsilon$-follows
$\mu$.

Let $q\in {\cal Q}^1(m_1,\dots,m_\ell;-m)$. Assume first that
the support $\mu$ of the vertical 
measured geodesic lamination of $q$ is large
of type $(m_1,\dots,m_\ell;-m)$. This is equivalent
to stating that $q$ does not have vertical
saddle connections. For $\epsilon >0$ 
let $\tau_\epsilon$ be a 
train track which carries $\mu$
and $\epsilon$-follows $\mu$. 
If $\epsilon >0$ is sufficiently small then
a carrying map $\mu\to \tau_{\epsilon}$ defines a 
bijection of the complementary components of $\mu$ onto
the complementary components of $\tau_\epsilon$. 
The transverse measure on $\tau_\epsilon$ defined by the vertical
measured geodesic lamination of $q$ is positive. 

Let
$\tilde C\subset {\bf H}^2$ be a complementary component of 
the preimage of $\mu$ in the hyperbolic plane ${\bf H}^2$.
Then $\tilde C$ is an ideal polygon whose vertices 
decompose the ideal boundary $\partial {\bf H}^2$ into 
finitely many arcs $a_1,\dots,a_k$ ordered counter-clockwise in
consecutive order.
Since $q$ does not have vertical saddle connections, 
the discussion in the 
proof of Proposition \ref{structure} shows the following. Let $\ell$ 
be a leaf of the preimage in ${\bf H}^2$ of the support $\nu$ of the 
horizontal measured 
geodesic lamination of $q$. Then the two 
endpoints of $\ell$ in ${\bf H}^2$ either are both contained in the
interior of the same arc $a_i$ or in the interior of two adjacent arcs
$a_i,a_{i+1}$.
As a consequence, for sufficiently small $\epsilon$
the geodesic
lamination $\nu$ is carried by the dual 
bigon track $\tau_\epsilon^*$ of $\tau_\epsilon$ (see the characterization 
of the set of measured geodesic laminations carried by
$\tau_\epsilon^*$ in \cite{PH92}).
Moreover, for any two 
adjacent subarcs $a_i,a_{i+1}$ of $\partial {\bf H}^2$ cut out by $\tilde C$, 
the transverse measure of the set of all leaves of the preimage of 
$\nu$ connecting these sides is positive. 
Therefore for sufficiently small $\epsilon$,
the horizontal measured geodesic lamination $\nu$
of $q$ defines an interior
point of ${\cal V}^*(\tau_\epsilon)$.

Now the set of quadratic differentials $z$ so that 
the support of the horizontal measured geodesic lamination of $z$ 
is large of type $(m_1,\dots,m_\ell;-m)$ is dense in 
the strong stable manifold $W^{ss}_{\tilde {\cal Q},{\rm loc}}(q)$ of $q$.
The above reasoning shows that for 
such a quadratic differential $z$ and for 
sufficiently
small $\epsilon$, the horizontal measured geodesic lamination 
of $z$ is carried by $\tau_\epsilon^*$. But this just means that
$\tau_\epsilon\in {\cal L\cal L}(m_1,\dots,m_\ell;-m)$. 
Moreover, 
if $r>0$ is the total weight which the vertical
measured geodesic lamination  puts on  
$\tau_\epsilon$ then
$\Phi^{-\log r}q$ is an interior point of 
${\cal Q}(\tau_\epsilon)$.
Thus $\tau_\epsilon$
satisfies the requirement in the proposition. 
Note that $\tau_\epsilon$ is necessarily non-orientable.

If $q\in {\cal H}^1(k_1,\dots,k_s)$ is such that the
support of the 
vertical measured geodesic lamination of $q$ is 
large of type $(2k_1,\dots,2k_s;0)$ then the above
reasoning also applies and yields an oriented large train track
with the required property.

Consider next the case that the support $\mu$ of 
the vertical measured geodesic lamination
of $q$ fills up $S$  but is 
not of type $(m_1,\dots,m_\ell;-m)$. Then $q$ has a vertical 
saddle connection. The set of all vertical saddle connections of $q$ is
a finite disjoint union $T$ of finite trees. 
The number of edges of 
this union of trees is uniformly bounded.
For $\epsilon>0$ let $\tau_\epsilon$ be a train track
which $\epsilon$-follows $\mu$ and carries $\mu$. 
If $\epsilon$ is sufficiently small then a carrying map
$\mu\to \tau_\epsilon$ defines a 
bijection between the complementary
components of $\mu$ and the complementary components of 
$\tau_\epsilon$ which induces a bijection between
their sides as well.

Modify $\tau_\epsilon$ as follows. Up to isotopy, 
a vertical saddle connection $s$
of $q$ is contained in a complementary component $C_s$ of $\tau_\epsilon$
which corresponds to the complementary component
of $\mu$ determined by $s$ 
(see the proof of Proposition \ref{structure}). Since
a carrying map $\mu\to \tau$ determines a bijection 
between the sides of the complementary components of $\mu$ and  the
sides of the complementary components of $\tau$,
the horizontal lines crossing through $s$ determine two non-adjacent sides
$c_1,c_2$ of $C_s$ (see once more
the discussion in the proof of Proposition \ref{structure}). Choose  
an embedded rectangle $R_s\subset C_s$ whose boundary
intersects the boundary of $C_s$ in two opposite sides 
contained in the interior of the sides $c_1,c_2$ of $C_s$. Up to an 
isotopy we may assume that 
these rectangles $R_s$ where $s$ runs through the vertical 
saddle connections of $q$
are pairwise disjoint. Collapse each of the
rectangles $R_s$ to a single segment in such a way that the 
two sides of $R_s$ which are contained in $\tau_\epsilon$ are identified
and form a single large branch $b_s$ as shown in Figure C. 
The branch $b_s$ can be isotoped to the saddle connection $s$. 
Let $\eta$ be the train track constructed in this way. Then $\eta$
is of topological type $(m_1,\dots,m_\ell;-m)$.
\begin{figure}[ht]
\begin{center}
\psfrag{Te}{$\tau_\epsilon$}
\psfrag{T1}{$\eta$}
\psfrag{q}{$q$}
\psfrag{s}{$s$}
\psfrag{Figure B}{Figure C} 
\includegraphics[width=0.9\textwidth]{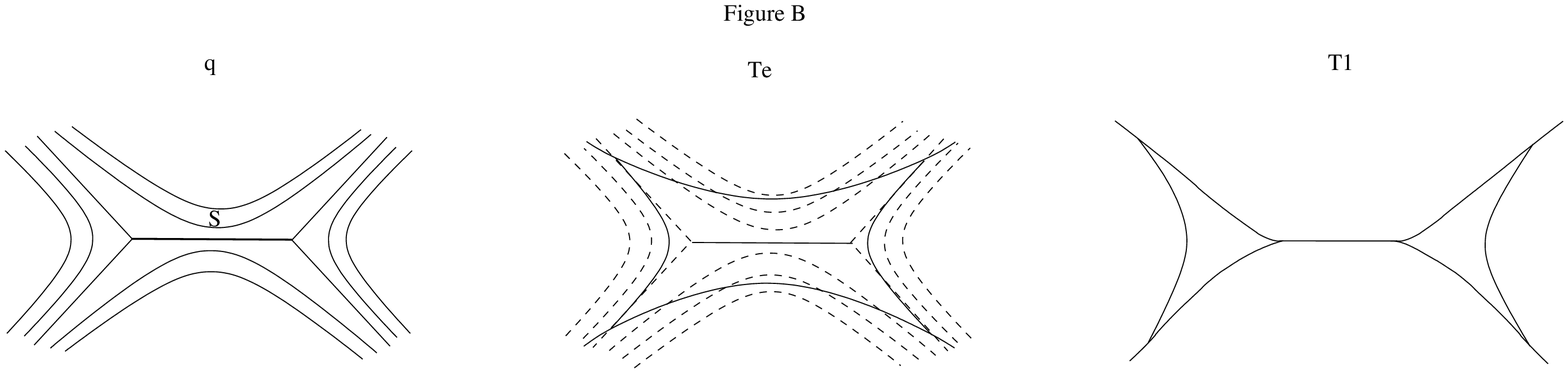}
\end{center}
\end{figure}

The train track $\tau_\epsilon$ can be obtained from $\eta$ by splitting
$\eta$ at each of the large branches $b_s$ and removing the 
diagonal of the split. In particular, $\eta$ carries $\tau_\epsilon$ and hence $\mu$.
The transverse measure on $\eta$ defined by the vertical
measured geodesic lamination of $q$ is positive
and consequently $\eta$ is recurrent. Moreover,  for sufficiently small 
$\epsilon$, the horizontal  measured geodesic lamination 
of $q$ is carried by  $\eta^*$. As above, we conclude that if $\epsilon>0$
is sufficiently small then 
$\eta$ is fully transversely recurrent and in fact large.
There is a  tangential measure on $\eta$ which is defined by the 
horizontal measured geodesic lamination of $q$ and
which gives positive weight to each of the branches $b_s$. Thus by possibly
decreasing once more the size of $\epsilon$, 
we can guarantee that 
for some $t\in \mathbb{R}$ the quadratic differential $\Phi^tq$ is an
interior point of ${\cal Q}(\eta)$. 
As a consequence, $\eta$ satisfies the
requirements in the proposition.

If the support  $\mu$ of the 
vertical measured geodesic lamination of $q$
is arbitrary then we proceed in the same way. Let $\epsilon >0$ be
sufficiently small that there is a
bijection between the complementary components of the train track
$\tau_\epsilon$ and the complementary
components of 
the support of $\mu$.
As before, we use the horizontal measured
foliation of $q$ to construct for every vertical saddle connection
$s$ of $q$ an embedded rectangle $R_s$ in $S$ whose interior
is contained in a complementary component of $\tau_\epsilon$ and
with two opposite sides on $\tau_\epsilon$ in such a way that 
the rectangles $R_s$ are pairwise disjoint. Collapse each of the rectangles
to a single arc. The resulting
train track has the required properties.

We discuss in detail the case that the support of $\mu$ contains a 
simple closed curve component 
$\alpha$. 
Then $\tau_\epsilon$ contains $\alpha$ as a simple closed curve
component as well. 
There is a vertical flat cylinder $C$ for $q$ foliated by smooth circles
freely homotopic to $\alpha$. The boundary $\partial C$ of $C$ 
is a finite union of vertical saddle connections. Some of these
saddle connections may occur twice on the boundary of $C$
(if $\mu=\alpha$ then this holds true for each of these saddle connections).
Assume without loss of generality
(i.e. perform a suitable isotopy) that $\alpha$ is a closed vertical
geodesic contained in the interior of $C$.

For each saddle connection $s$ in the boundary of $C$ 
choose a compact 
arc $a_s$ contained in the interior of $s$. 
Choose moreover a foliation ${\cal F}$ of $C$
by compact arcs with endpoints on the boundary of $C$ which is
transverse to the foliation of $C$ by
the vertical closed geodesics
and such that the following holds true. If $u_1,u_2$ are two distinct 
half-leaves of ${\cal F}$ with one endpoint in the arc $a_s$
and the second endpoint on $\alpha$ 
then the endpoints on $\alpha$ of $u_1,u_2$ 
are distinct. In particular, each arc $a_s$ which occurs
twice in the boundary of the cylinder $C$ determines an embedded
rectangle $R_s$ in $S$. Two opposite sides of $R_s$ are 
disjoint subarcs of $\alpha$; we call these
sides the vertical sides.  Each of the other two opposite sides consists
of two half-leaves of the foliation ${\cal F}$ which begin at a boundary
point of $a_s$ and end in 
a point of $\alpha$. The interior of the arc $a_s$ is contained in the
interior of $R_s$. 

The rectangles $R_s$ are pairwise disjoint. Therefore each of the rectangles
$R_s$ 
can be collapsed in $S$ to the arc $a_s$.
The resulting graph is a train track which carries 
$\alpha$ and contains for every 
saddle connection $s$ which occurs twice in the boundary of $C$ a large
branch $b_s$. 

If $s$ is a saddle connection 
on the boundary of $C$ which separates $C$ from 
$S-C$ then the arc $a_s$ is contained in 
the interior of a rectangle $R_s$ with 
one side contained in $\alpha$ and the second side contained in 
the interior of a branch of the
component of $\tau_\epsilon$ different from $\alpha$.
This branch is determined by the horizontal geodesics which cross through $s$.
As before, the rectangle $R_s$ is collapsed to a single branch.

To summarize, the train track $\tau_\epsilon$ can be modified
in finitely many steps to a train track $\eta$ with the required 
properties by collapsing 
for every vertical saddle connection of $q$ a rectangle with two sides
on $\tau_\epsilon$ to a single large branch.
This completes the construction and finishes the proof of the proposition.
\end{proof}

{\bf Remark:} In the proof of Lemma \ref{all}, we constructed
explicitly for every quadratic differential $q\in {\cal Q}(S)$ a 
train track $\tau_q$ belonging to the stratum of $q$. 
If $q$ is a one-cylinder Strebel differential then 
the train track $\tau_q$ is uniquely determined by the 
combinatorics of its vertical saddle connections   
on the boundary of the cylinder.
This fact in turn can be used to obtain a purely
combinatorial proof of the
classification results of Kontsevich-Zorich \cite{KZ03}  
and of Lanneau \cite{L08}.

Let again $\tau\in {\cal L\cal T}(m_1,\dots,m_\ell;-m)$.
Then $\tau\in {\cal L\cal T}(\tilde {\cal Q})$ for a component 
$\tilde {\cal Q}$ of ${\cal Q}^1(m_1,\dots,m_\ell;-m)$.
For every $\mu\in {\cal V}_0(\tau)$ and
every $\nu\in {\cal V}^*(\tau)$ so that the pair $(\mu,\nu)$ jointly
fills up $S$ there is a unique $q\in {\cal Q}(\tau)$ 
with vertical measured geodesic lamination $\mu$ and
horizontal measured geodesic lamination 
$\iota(\mu,\nu)^{-1}\nu$. Thus if 
$P{\cal V}^*(\tau)$ denotes the projectivization of the cone
${\cal V}^*(\tau)$ then for all $a<b$ 
there is a natural homeomorphism $\psi$ from 
the subset of ${\cal V}_0(\tau)\times P{\cal V}^*(\tau)\times [a,b]$
corresponding to 
pairs $(\mu,[\nu])$ which jointly fill up $S$ onto 
$C=\cup_{t\in [a,b]}\Phi^t{\cal Q}(\tau)$. 
The set $C$ is the closure in ${\cal Q}^1(S)$ of an open
subset of $\tilde Q$.
We say that the map $\psi$ 
defines on $C$ a \emph{train track product structure}.   
If $A\subset {\cal V}_0(\tau),B\subset P{\cal V}^*(\tau)$ 
are Borel sets then we also say that the image
of $A\times B\times [a,b]$ 
under the map $\psi$ has a train track product
structure. If $q\in {\cal Q}(\tau)$ and if $C$ is a neighborhood of 
$q$ with a train track product structure which is precisely invariant
under the stabilizer of $q$ in ${\rm Mod}(S)$ then we say that 
the projection of $C$ to ${\cal Q}(S)$ has a 
train track product structure.

The following proposition establishes product coordinates near boundary
points of strata. For this let again ${\cal Q}$ be a component of 
the stratum ${\cal Q}(m_1,\dots,m_\ell;-m)$, 
with closure $\overline{\cal Q}$. Let $\lambda$ be the
Lebesgue measure on ${\cal Q}$.

\begin{proposition}\label{forwardnondiv}
For every $q\in \overline{\cal Q}-{\cal Q}$ and every 
closed neighborhood $A$ of $q$ in $\overline{\cal Q}$  
there  is a closed neighborhood 
$K\subset A$ of $q$ in $\overline{\cal  Q}$ with the following properties.
\begin{enumerate}
\item
$K=\cup_{i=1}^kK_i$ for some $k\geq 1$, and $\lambda(K_i\cap K_j)=0$
for $i\not=j$.
\item For each $i$, the set $K_i$ contains
$q$ and has a
train track product structure. 
\end{enumerate}
\end{proposition}
\begin{proof} Our goal is to show that every point 
$q\in \overline{\cal Q}-{\cal Q}$  
has a closed
neighborhood $W$ in $\overline{\cal Q}$ 
with the following property. Let $\tilde {\cal Q}\subset 
{\cal Q}^1(S)$ be
a connected component of the preimage of ${\cal Q}$ 
and let $\tilde q$ be a 
lift of $q$ contained in the closure of $\tilde{\cal Q}$.
Then $W$ lifts to a 
contractible neighborhood $\tilde W$ of 
$\tilde q$ in the closure of $\tilde{\cal Q}$ 
which is precisely invariant under 
${\rm Stab}(\tilde q)$. Moreover, $\tilde W$ is 
contained in 
\[\cup_{j=1}^k\bigcup_{t\in [a_j,b_j]}\Phi^t{\cal Q}(\eta_j)\] 
for some $a_j<b_j$ 
where $\eta_j\in {\cal L\cal T}(\tilde {\cal Q})$
and where $\tilde q$ is contained
in the boundary of $\Phi^{-s_j}{\cal Q}(\eta_j)$ for some $s_j\in 
(a_j,b_j)$ $(j=1,\dots,k)$ .
For $i\not=j$ we have 
\begin{equation}\label{measureestimate}
\lambda\bigl(\bigcup_{t\in [a_i,b_i]}\Phi^t{\cal Q}(\eta_i)\cap 
\bigcup_{t\in [a_j,b_j]}\Phi^t{\cal Q}(\eta_j)\bigr)=0.\end{equation}

For this assume that 
$q\in {\cal Q}(n_1,\dots,n_s;-m)$ for some
$s<\ell$. Assume moreover for the moment 
that $q$ does not have vertical saddle connections.

Let $(q_i)\subset {\cal Q}$ be a sequence converging to $q$.
Since the subset of ${\cal Q}$ of 
quadratic differentials without
vertical saddle connection is dense in ${\cal Q}$, 
we may assume that for each $i$, $q_i$ does not
have a vertical saddle connection. 
Let $\tilde q_i\in \tilde{\cal Q}$ be a preimage of $q_i$ such that
$\tilde q_i\to \tilde q$.
For each $i$ the 
support $\mu_i$ of the vertical measured  
geodesic lamination of
$\tilde q_i$ is large of type $(m_1,\dots,m_\ell;-m)$.

We claim that up to passing to a subsequence,
the geodesic laminations $\mu_i$
converge in the Hausdorff topology to a large 
geodesic lamination $\xi$ of topological  type $(m_1,\dots,m_\ell;-m)$.
The lamination $\xi$ then contains the support $\nu$ of the vertical
measured geodesic lamination of $\tilde q$ as a sublamination.
Since $q$ does not have
vertical saddle connections, $\nu$ fills up $S$ and
$\xi$ can be obtained from $\nu$ by
adding finitely many isolated leaves.
These isolated leaves subdivide
some of the complementary components of $\nu$. 
The number of such limit laminations is uniformly bounded.

To see that this claim indeed holds true it is enough to
assume that $s=\ell-1$ and that $n_u=m_j+m_p$ for some
$j<p\leq \ell$ and some $u\leq s$ \cite{MZ08}- the purpose
of this assumption for our argument is to 
simplify the notations. Then 
for each sufficiently large $i$ 
the quadratic differential $\tilde q_i$ has a saddle connection $s_i$ connecting
a zero $x_1^i$ of order $m_j$ to 
a zero $x_2^i$ of order $m_p$ whose length 
(measured in the singular euclidean metric
defined by $\tilde q_i$) tends to zero as $i\to \infty$. 
More precisely, the
saddle connections $s_i$ converge to a zero $x_0$ of $\tilde q$ of 
order $n_u\geq 2$.
The length of any other saddle connection of $\tilde q_i$
is bounded from below
by a universal positive constant. 

Since $\tilde q_i$ 
does not have vertical saddle connections, locally near $x_1^i$ 
the interior of the saddle connection $s_i$ is contained in the interior of 
an euclidean sector based at $x_1^i$ of angle
$\pi$ bounded by two vertical separatrices $\alpha_1^i,\alpha_2^i$
of $\tilde q_i$ which issue from $x_1^i$. The union 
$\alpha_i=\alpha_1^i\cup \alpha_2^i$ is a smooth vertical geodesic line
passing through $x_1^i$, i.e. a
geodesic 
which is a limit in the compact open topology of geodesic segments
not passing through a singular point.
There are two 
vertical separatrices 
$\beta_1^i,\beta_2^i$ issuing from $x_2^i$ so that the sum of the angles
at $x_1^i,x_2^i$ of the (local) strip bounded by $\alpha_1^i,s_i,\beta_1^i$
equals $\pi$ and that the same holds true for the angle sum of the 
(local) strip bounded by $\alpha_2^i,s_i,\beta_2^i$. 
The vertical length of $s_i$ is positive. The union $\beta_i=\beta_1^i\cup 
\beta_2^i$ is a smooth vertical geodesic line passing through $x_2^i$.
 
Equip $S$ with the marked hyperbolic metric $P\tilde q\in {\cal T}(S)$.
For each $i$ 
lift the singular euclidean metric on $S$ defined by $\tilde q_i$ 
to a $\pi_1(S)$-invariant singular
euclidean metric on the universal covering ${\bf H}^2$ of $S$.
Let $\tilde s_i$ be a lift of the saddle connection $s_i$. Since $\tilde s_i$
is not vertical, 
the leaves of the vertical foliation of $\tilde q_i$ 
which pass through $\tilde s_i$ 
define a strip
of positive transverse measure in ${\bf H}^2$. This strip is bounded by
the two lifts $\tilde \alpha_i,\tilde \beta_i$ of 
the smooth vertical geodesics $\alpha_i,\beta_i$  
which pass through
the endpoints of $\tilde s_i$. 
As $i\to \infty$,
up to normalization and by perhaps passing to a subsequence, 
the vertical geodesics $\tilde\alpha_i,\tilde \beta_i$ converge 
in the compact open topology 
to vertical geodesics $\tilde\alpha,\tilde\beta$ for the singular euclidean metric
defined by
$\tilde q$ which pass through  
a preimage $\tilde x_0$ of the zero $x_0$ of $\tilde q$ of order $n_u=m_j+m_p\geq 2$.
By construction, the geodesics $\tilde \alpha,\tilde \beta$
coincide in a neighborhood of $\tilde x_0$. Since
there are no vertical saddle connections for $\tilde q$, we
necessarily have $\tilde \alpha=\tilde \beta$. Let 
$\tilde \gamma\subset {\bf H}^2$
be the hyperbolic geodesic  
with the same endpoints as $\tilde \alpha$ in the ideal boundary
of ${\bf H}^2$ (see \cite{L83} and the proof of Propositsion 
\ref{structure}). The projection of $\tilde \gamma$ to $S$ 
subdivides the complementary component
of $\nu$ containing $x_0$ into two ideal polygons with
$m_j+2$ and $m_p+2$ sides, respectively. 
The union of $\nu$ with this
geodesic is a large geodesic lamination $\xi$ of type $(m_1,\dots,m_\ell;-m)$.
This lamination is the limit in the Hausdorff topology of the laminations $\mu_i$.

Let $\xi_1,\dots,\xi_k\in {\cal L\cal L}(m_1,\dots,m_\ell;-m)$ be the 
(finitely many) large geodesic laminations obtained in this way. 
Each of the laminations $\xi_s$ contains $\nu$ as a sublamination, and
it is determined by a decomposition of a complementary $n_u+2$-gon of
$\nu$ into an ideal $m_j+2$-gon and an ideal $m_p+2$-gon. 
The set $\xi_1,\dots,\xi_k$ is invariant under the action of 
${\rm Stab}(\tilde q)$.
For sufficiently small $\epsilon>0$, a train track $\eta_j$ which carries $\xi_j$
and $\epsilon$-follows $\xi_j$ 
(for the hyperbolic metric $P\tilde q$) 
is a simple extension of a train track $\tau$ 
which carries $\nu$ and $\epsilon$-follows $\nu$.
The added branch is a diagonal of the complementary
$m_j+m_p+2$-gon of 
$\tau$ defined by the zero $x_0$ of $\tilde q$ of order $m_j+m_p$. It
decomposes this component  
into an $m_j+2$-gon and  an $m_p+2$-gon in a 
combinatorial pattern determined by $\xi_j$. 
The vertical measured geodesic lamination $\nu$ of $\tilde q$ 
defines a transverse measure on $\eta_j$ which gives 
full mass to the subtrack $\tau$ and hence it is contained in the
boundary of the cone ${\cal V}(\eta_j)$. 
We also may assume that the horizontal measured
geodesic lamination of $\tilde q$ is carried by the dual bigon track
$\eta_j^*$ (compare the proof of Lemma \ref{all}) and that the
set $\eta_1,\dots,\eta_k$ is invariant under the action of 
${\rm Stab}(\tilde q)$.

Since the set of 
geodesic laminations carried by a train track  is open and closed in 
the Hausdorff topology \cite{H09a},  for each $j$ the train track 
$\eta_j$ carries a minimal large geodesic lamination of 
type $(m_1,\dots,m_\ell;-m)$ (namely, the support of the 
vertical measured
geodesic lamination of a quadratic differential $\tilde q_i\in \tilde {\cal Q}$
sufficiently close to $\tilde q$
from the sequence which determines $\eta_j$) 
and hence it follows as in the
proof of Proposition \ref{structure} that $\eta_j\in {\cal L\cal T}(\tilde {\cal Q})$.
Moreover, if $s_j\in \mathbb{R}$ is such that $\Phi^{s_j}\tilde q\in 
{\cal Q}(\eta_j)$ then for every $\epsilon >0$ the set 
$\cup_j\cup_{t\in [-s_j-\epsilon,-s_j+\epsilon]}\Phi^t{\cal Q}(\eta_j)$ is 
a closed neighborhood of $\tilde q$ in the closure of 
$\tilde {\cal Q}$.

Now if 
$i\not=j$ then 
${\cal V}(\eta_i)\cap {\cal V}(\eta_j)={\cal V}(\tau)$ and hence this
intersection is contained in an affine subspace of codimension one.
Since the measure class of the conditional measures 
$\lambda^u$ of $\lambda$ coincides with the Lebesgue measure class
defined by the linear coordinates for the cone ${\cal V}(\eta_j)$, 
the equation (\ref{measureestimate}) holds true.

As a consequence, 
for suitable numbers $a_j<b_j$, the set
\[\cup_j\cup_{t\in [a_j,b_j]}\Phi^t{\cal Q}(\eta_j)\]
is a ${\rm Stab}(\tilde q)$-invariant closed neighborhood of $\tilde q$ in the closure
of $\tilde{\cal Q}$. 
In other words, 
there is  a ${\rm Stab}(\tilde q)$-invariant
finite collection of closed sets with 
train track product structures which cover a neighborhood of $\tilde q$ 
in the closure of $\tilde{\cal Q}$ 
and contain $\tilde q$ in their boundary. 
This completes the proof of the proposition in the case that 
the support of the vertical measured geodesic lamination of $\tilde q$
is large of type $(n_1,\dots,n_s;-m)$.

If the support of the vertical measured geodesic lamination of $\tilde q$ is
not large of type $(n_1,\dots,n_s;-m)$ then we argue in the same way.
In this case $\tilde q$ has a vertical 
saddle connection whose horizontal length is positive. 
Consider the action of the group $SO(2)$ 
on the space of quadratic differentials by rotation.
There is a sequence $\theta_j\in (0,\pi/2)$ with
$\theta_j\to 0$ such that the 
quadratic differential $e^{i\theta_j}\tilde q$ does not have
any vertical or horizontal saddle connection. Then  
the supports of the 
horizontal and the vertical
measured geodesic laminations of
$e^{i\theta_j}\tilde q$ are large of type $(n_1,\dots,n_s;-m)$.

Let $\tau\in {\cal L\cal T}(n_1,\dots,n_s;-m)$ be a train track
as in Lemma \ref{all} 
so that for some $\sigma>0$, 
$\Phi^\sigma \tilde q$ is an interior point of ${\cal Q}(\tau)$.
For sufficiently small $\theta$, say whenever 
$0<\vert \theta\vert <\epsilon$, we have
$e^{i\theta}\tilde q\in \cup_{s\in [\sigma-b,\sigma+b]}\Phi^s{\cal Q}(\tau)$ 
 where $b>0$ is a fixed number. If $\theta\in (-\epsilon,\epsilon)$ 
is such that
 $e^{i\theta}\tilde q$ does not have
 any vertical saddle connection then the argument in the beginning of 
 this proof shows that up to passing to a subsequence, for sufficiently 
 large $j$ the vertical measured geodesic lamination of $e^{i\theta}\tilde q_j$
 is carried by a simple extension of $\tau$ which 
 is large of type $(m_1,\dots,m_\ell;-m)$. As before, there are only
 finitely many such simple extensions, and these
 simple extensions define train track coordinates 
 on a neighborhood of $\tilde q$ in the closure of 
 $\tilde{\cal Q}$ as before.
 From this the proposition follows.
 \end{proof}

As an immediate consequence, we obtain the following.
Let $q\in \overline{\cal Q}-{\cal Q}$
and let 
$K=\cup_{i=1}^kK_i$ be as in Proposition \ref{forwardnondiv}.
Then for each $i\leq k$ there is an open subset $U_i\subset K_i$ of 
a strong unstable submanifold of ${\cal Q}$ whose  closure $A_i$ 
contains $q$. The set
\begin{equation}\label{wsuloc}
W^{su}_{{\cal Q},{\rm loc}}(q)=\cup_iA_i.\end{equation} 
is a compact subset of $W^{su}(q)$ which 
contains the intersection with $W^{su}(q)$ of every sufficiently small
neighborhood of $q$ in $\overline{\cal Q}$.  
Moreover, $\lambda^{su}(A_i\cap A_j)=0$ for $i\not=j$.

\section{Absolute continuity}

Let again
${\cal Q}$ be a connected component of a stratum
in ${\cal Q}(S)$. Then ${\cal Q}$ is invariant under the 
Teichm\"uller flow $\Phi^t$. 
For a periodic orbit
$\gamma\subset {\cal Q}$ for $\Phi^t$, the
Lebesgue measure supported in $\gamma$
is a $\Phi^t$-invariant Borel measure $\sigma(\gamma)$ 
on ${\cal Q}$ whose total mass equals the prime period 
$\ell(\gamma)$ of $\gamma$.
If we denote for $R>0$ by $\Gamma(R)$ the 
set of all periodic orbits for $\Phi^t$ of period at most
$R$ which are contained in ${\cal Q}$ then  
we obtain a finite $\Phi^t$-invariant 
Borel measure $\mu_R$ on ${\cal Q}$
by defining 
\begin{equation}\label{bowenmeasure}
\mu_R=e^{-hR}\sum_{\gamma\in \Gamma(R)}\sigma(\gamma).
\end{equation}

Let $\mu$ be any weak limit of the measures $\mu_R$ as
$R\to \infty$.
Then $\mu$ is a $\Phi^t$-invariant Borel measure on
${\cal Q}(S)$ supported in the closure $\overline{\cal Q}$ of ${\cal Q}$ 
(which may a priori be zero or 
locally infinite).
The purpose of this section is
to show

\begin{proposition}\label{absolute}
The measure $\mu$ on $\overline{\cal Q}$ satisfies $\mu\leq \lambda$.
\end{proposition}

This means that $\mu(U)\leq \lambda(U)$ for every
open relative compact subset $U$ of $\overline{\cal Q}$. In particular,
the measure $\mu$ is finite and absolutely continuous with respect to the
Lebesgue measure, and it gives full mass to ${\cal Q}$.

A point $q\in {\cal Q}$ is called \emph{forward recurrent}
(or \emph{backward recurrent}) if it is contained
in its own $\omega$-limit set (or in its own
$\alpha$-limit set) under the action of 
$\Phi^t$. A point $q\in {\cal Q}$ is \emph{recurrent}
if it is forward and backward recurrent. The set ${\cal R}\subset {\cal Q}$
of recurrent points is a $\Phi^t$-invariant Borel subset of ${\cal Q}$.
It follows from the work of Masur \cite{M82}
that a forward recurrent point
$q\in {\cal Q}(S)$ has uniquely ergodic
vertical and horizontal measured geodesic laminations 
whose supports fill up $S$. As a consequence,
the preimage $\tilde {\cal R}$ of 
${\cal R}$ in ${\cal Q}^1(S)$ is
contained in the set $\tilde{\cal A}$ defined in (\ref{tildea}) of 
Section 2.

Using the notations from Section 2, 
there is a number $p>1$ such that
for every $q\in {\cal Q}^1(S)$ the map
$t\to \Upsilon_{\cal T}(P\Phi^tq)$ is an unparametrized
$p$-quasi-geodesic in the curve graph ${\cal C}(S)$. 
If $q$ is a lift of a recurrent point in 
${\cal Q}(S)$ then this unparametrized quasi-geodesic is of infinite 
diameter. 

Recall from (\ref{distance}) of Section 2 the definition of the
distances $\delta_x$ $(x\in {\cal T}(S))$
on $\partial{\cal C}(S)$ and of the sets $D(q,r)\subset \partial{\cal C}(S)$
$(q\in \tilde {\cal A},r>0)$.
The following lemma is a version of
Lemma 2.1 of \cite{H07b}.

\begin{lemma}\label{expansion}
There are numbers $\alpha_0>0,\beta>0,b>0$ 
with the following property. 
Let $q\in \tilde {\cal R}$ and for $s>0$ write
$\sigma(s)=d(\Upsilon_{\cal T}(Pq),
\Upsilon_{\cal T}(P\Phi^sq))$; then   
\[\beta e^{-b\sigma(s)}\delta_{P\Phi^sq} \leq \delta_{Pq}\leq 
\beta^{-1}e^{-b\sigma(s)}\delta_{P\Phi^sq}\text{ on }
D(\Phi^sq,\alpha_0).\]
\end{lemma}

The map $F:\tilde {\cal A}\to 
\partial {\cal C}(S)$ defined in Section 2
is equivariant under the action of the
mapping class group on $\tilde {\cal A}\subset{\cal Q}^1(S)$ and on
$\partial {\cal C}(S)$. In particular,
for $q\in \tilde {\cal A}$ and $r>0$ the set
$D(q,r)\subset \partial {\cal C}(S)$ is invariant 
under ${\rm Stab}(q)$, and the same holds true for 
$F^{-1}D(q,r)$.

Let $\tilde {\cal Q}\subset {\cal Q}^1(S)$ be a component of the preimage of 
${\cal Q}$ and let
${\rm Stab}(\tilde {\cal Q})<{\rm Mod}(S)$ be the stabilizer of 
$\tilde {\cal Q}$ in ${\rm Mod}(S)$.
The $\Phi^t$-invariant Borel probability measure 
$\lambda$ on ${\cal Q}$ in the Lebesgue measure class 
lifts to a ${\rm Stab}(\tilde {\cal Q})$-invariant locally finite measure
on $\tilde {\cal Q}$ which we denote again by
$\lambda$.  The conditional measures $\lambda^{ss},
\lambda^{su}$ of $\lambda$ on the leaves of the 
strong stable and strong unstable foliation of ${\cal Q}$ lift to 
a family of conditional
measures on the leaves of the 
strong stable and strong unstable foliation
$W_{\tilde {\cal Q}}^{ss},W_{\tilde {\cal Q}}^{su}$ of $\tilde {\cal Q}$,  
respectively, which we
denote again by $\lambda^{ss},\lambda^{su}$ (see the discussion in Section 4).

\begin{lemma}\label{borelregular}
For every $\tilde q\in \tilde{\cal Q}\cap \tilde {\cal R}$
and for all 
compact neighborhoods $W_1\subset W_2$ of $\tilde q$ 
in $W^{su}_{\tilde {\cal Q}}(\tilde q)$ 
there are compact  neighborhoods
$K\subset C\subset W_1$ of 
$\tilde q$ in $W_{\tilde{\cal Q}}^{su}(\tilde q)$ 
with the following properties.
\begin{enumerate}
\item $K,C$ are precisely invariant under ${\rm Stab}(\tilde q)$.
\item There are numbers $0<r_1<r_2<\alpha_0/2$ such that
\[K=\overline{W_1\cap F^{-1}D(\tilde q,r_1)},\,
C=\overline{W_1\cap F^{-1}D(\tilde q,r_2)}.\]
\item $\lambda^{su}(K)(1+\epsilon)\geq \lambda^{su}(C)$.
\item If $z\in K\cap \tilde{\cal  A}$ and if 
$Z\subset \overline{F^{-1}D(z,(r_2-r_1)/2)\cap 
W_2}$ then $Z\subset C$.
\end{enumerate}
\end{lemma}
\begin{proof} Let $q\in {\cal Q}$  be a recurrent point and let  
$\tilde q\in \tilde {\cal Q}$ be a lift of $q$. 
Let $W_1\subset W_2\subset W_{\tilde {\cal Q}}^{su}(\tilde q)$ 
be compact  neighborhoods of $\tilde q$
and let $r>0$ be such that
$B_{\tilde{\cal Q}}^{su}(\tilde q,2r)\subset 
W_1\subset W_{\tilde {\cal Q}}^{su}(\tilde q)$ 
is precisely invariant under ${\rm Stab}(\tilde q)$ and 
projects to a metric
orbifold ball in $W^{su}_{\cal Q}(q)$.

By Lemma \ref{closed}, the 
map $F:\tilde {\cal A}\to \partial {\cal C}(S)$ 
is continuous and closed, and the sets 
$F(B^{su}(\tilde q,\nu)\cap \tilde {\cal A})$ 
$(\nu>0)$ form a neighborhood basis
of $F\tilde q$ in $\partial{\cal C}(S)$. 
Thus there is a number $u_0>0$ such that
\[D(\tilde q,u_0)\cap F({W_2}
\cap \tilde{\cal A})\subset 
F(B^{su}_{\tilde {\cal Q}}(\tilde q,r)\cap \tilde {\cal A}).\] 

For $u\leq u_0$ let 
$K_u\subset W^{su}_{\tilde {\cal Q}}(\tilde q)$ 
be the closure of the set
\[F^{-1}(D(\tilde q,u))\cap W_2.\]
Then $K_u$ is a closed neighborhood of $\tilde q$ in 
$W_{\tilde {\cal Q}}^{su}(\tilde q)$ which is 
precisely invariant under ${\rm Stab}(\tilde q)$. 
Moreover, 
$K_t\subset K_u$ for $t< u$, and 
Lemma \ref{closed} shows that $\cap_{u>0}K_u=\{\tilde q\}$.
Since the conditional measure
$\lambda^{su}$ on $W^{su}_{\tilde {\cal Q}}(\tilde q)$ is Borel regular, 
for every $\epsilon >0$ 
there are numbers $r_1<r_2<u_0$ so that
\[\lambda^{su}(K_{r_1})\geq \lambda^{su}(K_{r_2})(1+\epsilon)^{-1}.\]
For these number $r_1<r_2$, all requirements in the lemma
hold true. 
This shows the lemma.
\end{proof}

{\bf Remark:} Since $\tilde {\cal A}$ is dense in ${\cal Q}^1(S)$ and
the map $F:\tilde {\cal A}\to \partial {\cal C}(S)$ is continuous
and closed, the sets $K\subset C\subset W^{su}_{\tilde {\cal Q}}(\tilde q)$
have dense interior. Moreover, we may assume that their boundaries
have vanishing Lebesgue measures.

Let again $\tilde {\cal Q}\subset {\cal Q}^1(S)$ be a component of 
the preimage of ${\cal Q}$. 
For $q\in {\cal Q}$ let $\tilde q$ be a preimage of $q$ in 
$\tilde {\cal Q}$ and 
let $\vert {\rm Stab}(q)\vert$ be the cardinality
of the quotient of ${\rm Stab}(\tilde q)$ 
 by the normal subgroup of all elements of 
 ${\rm Stab}(\tilde q)$ which fix $\tilde {\cal Q}$ pointwise
 (for example, the hyperelliptic involution of a closed surface of 
 genus $2$ acts trivially on the entire bundle ${\cal Q}^1(S)$).
We note

\begin{lemma}\label{cals}
The set ${\cal S}=\{q\in {\cal Q}\mid 
\vert {\rm Stab}(q)\vert=1\}$ is an open dense $\Phi^t$-invariant
submanifold of ${\cal Q}$.
\end{lemma}
\begin{proof}
The mapping class group preserves the Teichm\"uller metric on
${\cal T}(S)$ and hence an element $h\in {\rm Mod}(S)$ which
stabilizes a quadratic differential $\tilde q\in {\cal Q}^1(S)$ fixes
pointwise the Teichm\"uller geodesic 
with initial cotangent $\tilde q$. Therefore the set ${\cal S}$ is $\Phi^t$-invariant, 
moreover it is clearly open. Since the Teichm\"uller
flow on ${\cal Q}$ has dense orbits, either
${\cal S}$ is empty or dense. However, ${\rm Mod}(S)$
acts properly discontinuously on ${\cal T}(S)$ and consequently the first
possibility is ruled out by the fact that the conjugacy class of an element of ${\rm Mod}(S)$ 
which fixes an entire component of the preimage of ${\cal Q}$ does not contribute
towards $\vert {\rm Stab}(q)\vert$.
\end{proof}

For a control of the measure $\mu$ 
we use a variant of an argument of Margulis \cite{Mar04}.
Namely, for numbers $R_1<R_2$ 
let $\Gamma(R_1,R_2)$ be the set of all periodic orbits of $\Phi^t$
which are contained in ${\cal Q}$, with prime periods
in the interval $(R_1,R_2)$. For 
an open or closed  subset $V$ of
$\overline{\cal Q}$ and numbers $R_1<R_2$ define 
\[H(V,R_1,R_2)=\sum_{\gamma\in \Gamma(R_1,R_2)}
\int_{\gamma}\chi(V)\]
where $\chi(V)$ is the characteristic function of $V$.

To obtain control on the quantities $H(V,R_1,R_2)$ 
we use a tool from \cite{ABEM10}.
Namely, every leaf $W^{ss}(q)$ of the strong stable foliation 
of ${\cal Q}(S)$ can
be equipped with the \emph{Hodge distance} $d_H$ 
(or, rather, the modified Hodge distance, \cite{ABEM10}).
This Hodge distance is defined by a norm on the tangent space
of $W^{ss}(q)$ (with a suitable interpretation). In particular,
closed $d_H$-balls of sufficiently small finite radius are compact,
and balls about a given point $q$ define a neighborhood basis
of $q$ in $W^{ss}(q)$.  
We also obtain a Hodge distance on the leaves of the
strong unstable foliation as the image under the flip ${\cal F}$
of the Hodge
distance on the leaves of the strong stable foliation.
These Hodge distances restrict to Hodge distances on the leaves
of the foliations $W^{ss}_{\cal Q},W^{su}_{\cal Q}$ which we denote
by the same symbol $d_H$.

The following result is Theorem 8.12 of \cite{ABEM10},

\begin{theorem}\label{hodgenorm}
There is a number
$c_H>0$ such that
\begin{equation}\label{hodge}
d_H(\Phi^tq,\Phi^t q^\prime)\leq c_Hd_H(q,q^\prime).
\end{equation}
for all 
$q\in {\cal Q}(S),q^\prime\in W^{ss}(q)$
and all 
$t\geq 0$.
\end{theorem}

The next lemma
provides some first volume control for the measure $\mu$.

\begin{lemma}\label{upperestimate}
For every recurrent point $q\in {\cal Q}$ with
$\vert {\rm Stab}(\tilde q)\vert =1$, for 
every neighborhood $V$ of $q$ in ${\cal Q}$ and 
for every $\epsilon >0$ there is 
a number $t_0>0$ and there is an  
open neighborhood $U\subset V$ of $q$ such that
\[\lim\sup_{R\to \infty}H(U,R-t_0,R+t_0)e^{-hR}\leq
2t_0\lambda(U)(1+\epsilon).\]
\end{lemma}
\begin{proof}
We use the strategy of the proof of Lemma 6.1 of \cite{Mar04}. 
The idea is to find for every 
recurrent point $q\in {\cal Q}$ with $\vert {\rm Stab}(\tilde q)\vert =1$,
for every neighborhood $V$ of $q$ in ${\cal Q}$  
and for every $\epsilon \in (0,1)$ some number $t_0>0$ and closed 
neighborhoods 
$Z_1\subset Z_2\subset Z_3\subset V_0\subset V$ of $q$ in ${\cal Q}$ 
with dense interior such that
for all sufficiently large $R>0$ the following properties hold.
\begin{enumerate}
\item $V_0$ is connected and has a local product structure.
\item $\lambda(Z_3)\leq
\lambda(Z_1)(1+\epsilon)$. 
\item Let $z\in Z_1$ and assume that $\Phi^\tau z=z$ for some
$\tau\in (R-t_0,R+t_0)$. Let $\hat E$ be the 
component containing $z$ of the 
intersection $\Phi^\tau V_0\cap V_0$ and let 
$E=\hat E\cap \Phi^\tau Z_2\cap Z_3$. Then 
\[\lambda(E)\in  [e^{-hR}\lambda(Z_1)/(1+\epsilon),
e^{-hR}\lambda(Z_1)(1+\epsilon)],\] and the length
of the connected orbit subsegment 
of $(\cup_{t\in \mathbb{R}}\Phi^tz)\cap Z_1$
containing $z$ equals $2t_0$.
\item There is at most one  
periodic orbit for $\Phi^t$ 
of prime period $\sigma\in (R-t_0,R+t_0)$ which intersects $E$,
and the intersection of this orbit with $E$ is connected.
 \end{enumerate}

The construction is as follows.

Let $q\in {\cal Q}$ 
be recurrent with $\vert {\rm Stab}(\tilde q)\vert=1$ 
and let $V$ be a
neighborhood of $q$ in ${\cal Q}$. Using the notations
from Subsection 4.1,  
for $\epsilon >0$ there are numbers $a_0<a_{\cal Q}(q),
t_0<\min\{t_{\cal Q}(q)/4(1+\epsilon),\log(1+\epsilon)/h\}$
such that 
\[V_0=V(B_{{\cal Q}}^{ss}(q,a_0),
B^{su}_{{\cal Q}}(q,a_0),t_0(1+\epsilon))\subset V\]
is a set with a local product structure.

Let $\tilde q\in {\cal Q}^1(S)$ be a preimage of $q$.
By construction (see the discussion in Section 4.1), the
set \[\tilde V_0=V(B^{ss}_{\cal Q}(\tilde q,a_0),
B_{\cal Q}^{su}(\tilde q,a_0),t_0(1+\epsilon))\] is 
precisely invariant under ${\rm Stab}(\tilde q)$.
In particular,  since $\vert {\rm Stab}(\tilde q)\vert =1$, the
set $\tilde V_0$ is mapped homeomorphically onto 
$V_0$ by the projection $\tilde {\cal Q}\to {\cal Q}$.

Since periodic orbits for $\Phi^t$ are in bijection
with conjugacy classes of pseudo-Anosov elements of 
${\rm Mod}(S)$, up to making $a_0$ smaller
we may assume that  the following holds true.
For every $r> 8t_0$, every component of the intersection
$\Phi^rV_0\cap V_0$ is intersected by at most one
periodic orbit for the Teichm\"uller flow 
with prime period contained in the interval $[r-2t_0,r+2t_0]$, and if 
such an orbit exists then its intersection with $\Phi^rV_0\cap V_0$ is 
connected.

As in (\ref{theta}) of Section 4, for $z\in V_0$ let
$\theta_z:B^{ss}_{\cal Q}(q,a_0)\to W^{ss}_{{\cal Q},{\rm loc}}(z)$ 
be defined by the requirement that $\theta_z(u)\in W^u_{{\cal Q},{\rm loc}}(u)$
for all $u$. Similarly, as in (\ref{zeta}) of Section 4, let 
$\zeta_z:B^{su}_{\cal Q}(q,a_0)\to W^{su}_{{\cal Q},{\rm loc}}(z)$
be defined by  $\zeta_z(u)\in W^s_{{\cal Q},{\rm loc}}(u)$.
We claim that for sufficiently small $a_1<a_0$ and
for every 
\[z\in V_1=V(B^{ss}_{\cal Q}(q,a_1),B^{su}_{\cal Q}(q,a_1),t_0)\]
the following holds true.
\begin{enumerate}
\item[a)] The Jacobian of the embedding 
 $\theta_z:B^{ss}_{\cal Q}(q,a_1)\to W^{ss}_{{\cal Q},{\rm loc}}(z)$ and
of the embedding  
$\zeta_z:B^{su}_{\cal Q}(q,a_1)\to W^{su}_{{\cal Q},{\rm loc}}(z)$ 
with respect to the measures $\lambda^{ss}$ and 
$\lambda^{su}$, respectively, 
is contained in the interval $[(1+\epsilon)^{-1},1+\epsilon]$.
\item[b)] The restriction to $V_1$ of the 
function $\sigma$ defined in (\ref{sigma}) 
takes values in the interval
$[-(\log(1+\epsilon)/h),(\log(1+\epsilon))/h]$.
\item[c)] If  $z\in V(B^{ss}_{\cal Q}(q,a_1),B^{su}_{\cal Q}(q,a_1))$ and 
if $t>8t_0$ is such that 
\[\Phi^tz\in V(B^{ss}_{\cal Q}(q,a_1),B^{su}_{\cal Q}(q,a_1))\] 
then $\Phi^t(V_1\cap W^s_{{\cal Q},{\rm loc}}(z))\subset
V_0$ and $\Phi^{-t}(V_1\cap W^u_{{\cal Q},{\rm loc}}(z))\subset V_0$.
\end{enumerate}
Here and in the sequel, for $z\in V_1$ we denote by 
$V_1\cap W^s_{{\cal Q},{\rm loc}}(z)$ the connected component
containing $z$ of the intersection $V_1\cap W^s_{\cal Q}(z)$.

To verify the claim, note 
first that property b) can be fulfilled since $\sigma$ is continuous 
and $\Phi^t$-invariant and equals one at $q$.
Property a) is fulfilled for sufficiently small $a_1$ since the 
measures $\lambda^s$ (or $\lambda^u$) are invariant under
holonomy along the strong unstable (or the strong unstable) foliation
and since $d\lambda^{s}=d\lambda^{ss}\times dt$ and 
$d\lambda^u=d\lambda^{su}\times dt$ and hence Jacobians
of the maps $\theta_z,\zeta_z$ are controlled by the function $\sigma$.

By Property b) above and by Theorem \ref{hodgenorm},
the last property is fulfilled if 
we choose $a_1>0$ small enough so that for some $r>0$ and 
every $z\in V_1$ 
the following is satisfied. For every $u\in V_1$ 
the diameter of $\theta_u(B^{ss}_{\cal Q}(q,a_1))$
with respect to the Hodge distance 
does not exceed $r$, and 
the Hodge distance between $\theta_u(B^{ss}_{\cal Q}(q,a_1))$ and
the boundary of $\theta_u(B^{ss}_{\cal Q}(q,a_0))$
is not smaller than $c_Hr$.

Since $h\geq 1$, 
Property b) implies the following. For all closed
sets $A^{i}\subset B^{i}_{\cal Q}(q,a_1)$ $(i=ss,su)$ and for every $z\in 
V(B^{ss}_{\cal Q}(q,a_1),B^{su}_{\cal Q}(q,a_1))$ we have
\begin{align}\label{exactcontrol}
%Y(K^{ss},K^{su},t_0)
% &\subset V(K^{ss},K^{su},t_0(1+\epsilon))\text{ and }\\ 
V(A^{ss},A^{su},t_0(1+\epsilon)^{-1})\subset 
V(\theta_z(A^{ss}),\zeta_z(A^{su}),t_0) \\
\subset V(A^{ss},A^{su},t_0(1+\epsilon)).\notag \end{align}
Moreover, we have
\begin{equation}\label{productestimate}
\lambda(V(\theta_z(A^{ss}),\zeta_z(A^{su}),t_0))/2t_0\lambda^{ss}(A^{ss})
\lambda^{su}(A^{su})\in [(1+\epsilon)^{-4},(1+\epsilon)^4].
\end{equation}

By the estimate (\ref{deltacomparison}) in 
Section 2, there is a number 
$\kappa>0$ such that for any two points $u,x\in {\cal T}(S)$ with
$d_{\cal T}(u,x)\leq 1$ the distances $\delta_u,\delta_x$ on 
$\partial {\cal C}(S)$ are $e^\kappa$-bilipschitz equivalent. 
 
Let
$\tilde{\cal Q}$ be a component of the preimage
of ${\cal Q}$ in ${\cal Q}^1(S)$.
Let $\tilde q\in \tilde{\cal Q}$ be a lift of $q$.
Choose
closed neighborhoods $K^{ss}\subset C^{ss}\subset 
B^{ss}_{\tilde {\cal Q}}(\tilde q,a_1)
\subset B^{ss}_{\tilde {\cal Q}}(\tilde q,a_0)$ of $\tilde q$ 
whose images under the flip ${\cal F}$ satisfy the properties 
in Lemma \ref{borelregular} for some numbers $0<r_1<r_2<\alpha_0/2e^\kappa$
where $\alpha_0>0$ is as in Lemma \ref{expansion}.
Choose also closed neighborhoods $\tilde K^{su}\subset 
\tilde C^{su}\subset 
B^{su}_{\tilde {\cal Q}}(\tilde q,a_1)
\subset B^{su}_{\tilde {\cal Q}}(\tilde q,a_0)$ of $\tilde q$
with the properties in Lemma \ref{borelregular} for some 
numbers $0<\tilde r_1<\tilde r_2<\alpha_0/2\kappa$.
By the choice of  the set $V_0$, 
for any two points  
$u,z\in V(C^{ss},\tilde C^{su},t_0(1+\epsilon))$ the 
distances $\delta_{Pu}$ and $\delta_{Pz}$ are
$e^\kappa$-bilipschitz equivalent. 
As a consequence,
for all $u\in V(C^{ss},\tilde C^{su},t_0(1+\epsilon))$ 
the $\delta_{Pu}$-diameter of $F({\cal F}C^{ss}\cap {\cal A})$ and
$F(\tilde C^{su}\cap {\cal A})$ does not exceed
$\alpha_0/2$. Let 
\[\rho_0\in (0,\min\{(r_2-r_1)/2,(\tilde r_2-\tilde r_1)/2\}).\]

By assumption, $q$ is recurrent and hence by Lemma \ref{expansion},
applied to both $\tilde q$ and $-\tilde q={\cal F}(\tilde q)$,  
there is a number $R_0>0$
so that for every $R\geq R_0$ and for every 
$z\in B^{su}_{\tilde {\cal Q}}(\tilde q,a_1)$ with 
$d_{\cal T}(P\Phi^Rz,P\Phi^R\tilde q) \leq 1$ we have
\begin{align}\label{lemmaupper}
\delta_{P\Phi^R z}\leq \rho_0 \delta_{Pz}/\alpha_0\text{ on }
F({\cal F}C^{ss}\cap \tilde{\cal A})  \text{ and }\\
\delta_{P\Phi^R z}\geq \alpha_0\delta_{Pz}/
\rho_0 \text{ on }
D(\Phi^R\tilde q,\alpha_0).
\notag\end{align} 
Moreover, there is a mapping class $h\in {\rm Stab}(\tilde {\cal Q})$ and
a number $R_1>R_0$ 
such that 
$\Phi^{R_1}\tilde q$ is an
interior point of $hV(K^{ss},\tilde K^{su})$.

By equivariance under the action of the mapping class group, 
for every $u\in hV(C^{ss},\tilde C^{su})$ the $\delta_{Pu}$-diameter
of $F(hV(C^{ss},\tilde C^{su})\cap \tilde {\cal A})$ 
is smaller than $\alpha_0/2$. In particular, the 
$\delta_{P\Phi^{R_1}\tilde q}$-diameter of $F(h\tilde C^{su}\cap 
\tilde{\cal A})$ is smaller than $\alpha_0/2$. 
The second part of inequality (\ref{lemmaupper}) then implies 
that the $\delta_{P\tilde q}$-diameter of 
$F(h\tilde C^{su}\cap \tilde{\cal A})$ does not exceed 
$\rho_0$. Thus  by Property c) above, 
by the choice of $\rho_0$ and by
Lemma \ref{borelregular}, 
we have
\[F(h\tilde C^{su}\cap \tilde {\cal A})\subset 
F(\tilde C^{su}\cap \tilde {\cal A}).\]
 
Define 
\begin{align} 
K^{su} &=\overline{\{x\in W^{su}_{\tilde {\cal Q},{\rm loc}}(\tilde q)
\cap \tilde {\cal A}\mid
F(x)\in F(h\tilde K^{su}\cap \tilde {\cal A})\}}\text{  and }\notag\\
C^{su} &=\overline{\{x\in W^{su}_{\tilde {\cal Q},{\rm loc}}(\tilde q)
\cap \tilde {\cal A}\mid 
F(x)\in F(h\tilde C^{su}\cap \tilde {\cal A})\}}.\notag
\end{align}
Then $\tilde q$ is an interior point of $K^{su}$ (as a subset
of $W^{su}_{\tilde {\cal Q},{\rm loc}}(\tilde q)$), and 
$K^{su},C^{su}$ are precisely invariant under ${\rm Stab}(\tilde q)$
(since a non-trivial element of ${\rm Stab}(\tilde q)$ fixes
$\tilde {\cal Q}$ pointwise).

The conditional measures $\lambda^{su}$ are
invariant under holonomy along the 
strong stable foliation and transform under the 
Teichm\"uller flow by $\lambda^{su}\circ \Phi^t=e^{ht}\lambda^{su}$.
Moreover, $\lambda^{su}(\tilde K^{su})\geq
\lambda^{su}(\tilde C^{su})(1+\epsilon)^{-1}$ and hence
properties a) and b) above and the definition of the
function $\sigma$ imply that 
\[\lambda^{su}(K^{su})\geq \lambda^{su}(C^{su})(1+\epsilon)^{-3}.\]

Define 
\begin{equation}
\tilde Z_1=V(K^{ss},K^{su},t_0),\,
 \tilde Z_2=V(K^{ss},C^{su},t_0),\,
\tilde Z_3=V(C^{ss},C^{su},t_0(1+\epsilon))\notag \end{equation}
and let $Z_i$ be the projection of $\tilde Z_i$ to ${\cal Q}$. 
Note that we have $Z_1\subset Z_2\subset Z_3$ and 
\[\lambda(Z_1)\geq \lambda(Z_3)(1+\epsilon)^{-8}\] by the 
choice of $K^{ss},C^{ss}$, by the estimate in a) above, by
invariance of $\lambda$ under the flow $\Phi^t$
(which implies that $\lambda(\tilde Z_3)\leq 
\lambda V(C^{ss},C^{su},t_0)(1+\epsilon)^2$) and 
by the fact that $\tilde Z_i$ is mapped homeomorphically onto $Z_i$
for $i=1,2,3$. Moreover, each of the sets $Z_i$ is closed with
dense interior.

Let $R>R_1+t_0$ and 
let $z\in Z_1$ be a periodic point for $\Phi^t$ of period 
$r\in [R-t_0,R+t_0]$. Since every orbit of $\Phi^t$ which intersects
$Z_1$ also intersects $V(K^{ss},K^{su})$ we may assume that
$z\in V(K^{ss},K^{su})$.
Let $\hat E$ be the component containing $z$ of the
intersection $\Phi^rV_0\cap V_0$ and let 
\[E=\hat E\cap \Phi^rZ_2\cap Z_3
\subset Z_3.\] 
We claim that 
\begin{equation}\label{lambdaestimate}
\lambda(E)\in 
[e^{-hr}\lambda(Z_1)
(1+\epsilon)^{-10},e^{-hr}\lambda(Z_1)
(1+\epsilon)^{11}].
\end{equation}

To see that this is indeed the case, let $\tilde z\in \tilde Z_1$
be a lift of $z$. By the choice of the set $C^{su}$ and 
by the first part of the estimate (\ref{lemmaupper}), the
$\delta_{P\Phi^R\tilde z}$-diameter of the set
$F({\cal F}
\Phi^RC^{ss}\cap \tilde {\cal A})$
does not exceed $\rho_0$. In particular, since $z\in Z_1$ 
and Property c) above holds true, we have 
\begin{equation}\label{contained}
\Phi^{r}(W^{s}_{{\cal Q},{\rm loc}}(z)\cap Z_2)\subset E\end{equation}
and similarly
\begin{equation}\label{contain}
\Phi^{-r}(W^u_{{\cal Q},{\rm loc}}(z)\cap Z_1)\subset E.\end{equation}

Let $D\subset C^{ss}$ be such that 
\[\theta_zD=
\Phi^r(W^s_{{\cal Q},{\rm loc}}(z)\cap Z_2)\cap \theta_zC^{ss}.\]
Then by
the estimate (\ref{exactcontrol}) and by
(\ref{contain}), we have 
\[Q_1=V(\theta_z(D),\zeta_z(K^{su}),t_0(1+\epsilon)^{-1})\subset E\subset
V(\theta_z(D),\zeta_z(C^{su}),t_0(1+\epsilon))=Q_2.\]
Now by the 
estimate (\ref{productestimate} and the fact that $\Phi^r$ 
preserves the stable foliation and contracts the 
measures $\lambda^{s}$ by the factor $e^{-hr}$,
we conclude that \[\lambda(Q_1)\geq 
e^{-hr}\lambda^{ss}(K^{ss})\lambda^{su}(K^{su})/2t_0(1+\epsilon)^{6}\]
and similarly 
\[\lambda(Q_2)
\leq 
e^{-hr}\lambda^{ss}(K^{ss})\lambda^{su}(C^{su})(1+\epsilon)^6/2t_0.\]
Together with the estimate (\ref{productestimate}) this 
implies the estimate (\ref{lambdaestimate}).

Since $r\in [R-t_0,R+t_0]$ and $e^{ht_0}\leq 1+\epsilon$
we conclude that the measure of the set
$E$ is at least
\[e^{-hR}\lambda(Z_1)(1+\epsilon)^{-12}.\]
Moreover, the Lebesgue measure of the component containing $z$ 
of the orbit segment $\{\Phi^tz\mid -t_0<t<t_0\}\cap Z_2$ 
equals $2t_0$.

On the other hand,
if $z\not=z^\prime\in Z_0$ are periodic
points of prime periods $r,s\in [R-t_0,R+t_0]$ then 
by our choice of $V_0$ 
the components containing $z,z^\prime$ of the intersection
$\Phi^rV_0\cap V_0$
are disjoint. Thus there are at most 
\[\lambda(\Phi^RZ_2\cap Z_3)e^{hR}(1+\epsilon)^{12} /\lambda(Z_1)\]
such intersection arcs which are subarcs of 
periodic orbits
of prime period in $[R-t_0,R+t_0]$. However, since the Lebesgue 
measure $\lambda$ is mixing for the Teichm\"uller flow
\cite{M82,V86}, for sufficiently large $R$ we have
\[\lambda(\Phi^RZ_2\cap Z_3)\leq 
\lambda(Z_2)\lambda(Z_3)(1+\epsilon)\leq \lambda(Z_1)^2(1+\epsilon)^{17}.\]
From this we deduce that
\[H(Z_1,R-t_0,R+t_0)e^{-hR}\leq {2t_0\lambda(Z_1)}(1+\epsilon)^{29}\] 
for all sufficiently large $R>0$.
This shows the lemma.
\end{proof}

Now we are ready for the proof of Proposition \ref{absolute}.

{\it Proof of Proposition \ref{absolute}.} 
Let $\mu$ be a weak limit of the measures $\mu_R$ as
$R\to \infty$. Then $\mu$ is a (a priori
locally infinite) $\Phi^t$-invariant
Borel measure supported in the closure 
$\overline{\cal Q}$ of ${\cal Q}$. This measure is moreover invariant
under the flip ${\cal F}:q\to -q$.

By Lemma \ref{upperestimate} it suffices to show the following. 
Let $A\subset \overline{\cal Q}$ be a closed
$\Phi^t$-invariant set of vanishing Lebesgue measure. 
Then for all $\epsilon>0$, every $q\in A$ has a neighborhood
$U$ in $\overline{\cal Q}$ such that
$\mu(A\cap U)<\epsilon$.

First let $q\in A\cap {\cal Q}$. Choose compact 
balls $B^i\subset C^i\subset W^i_{{\cal Q},{\rm loc}}(q)$
about $q$ for the Hodge distance
of radius $r_1>0,r_2>2c_Hr_1>0$ 
$(i=ss,su)$ and numbers
$t_0>0,\delta >0$ such that
$V_3=V(C^{ss},C^{su},t_0(1+\delta))$ is a set
with a local product structure.
In particular, for every preimage $\tilde q$ of $q$ in
${\cal Q}^1(S)$ the component of the preimage
of $V_3$ containing $\tilde q$ is
precisely invariant under ${\rm Stab}(\tilde q)$. 
Then 
\[V_0=V(B^{ss},B^{su},t_0(1-\delta))\subset 
V(C^{ss},C^{su},t_0(1+\delta))=V_3\]
are closed neighborhoods of $q$
in ${\cal Q}$. Let moreover
\[V_1=V(B^{ss},B^{su},t_0)\subset V_2=V(B^{ss},C^{su},t_0).\]
We may assume that for one (and hence every) component 
$\tilde V_3$ of the
preimage of $V_3$ in ${\cal Q}^1(S)$ the diameter of the projection 
$P\tilde V_3$ of $\tilde V_3$ to ${\cal T}(S)$ does not exceed one.
As in the proof of Lemma \ref{upperestimate} we require that
moreover the following holds true. 

$(*)$  If $z\in V(B^{ss},B^{su})$ and if 
$t>8t_0$ is such that $\Phi^tz\in V(B^{ss},B^{su})$ then 
$\Phi^t(V_1\cap W^s_{{\cal Q},{\rm loc}}(z))\subset
V_3$ and moreover
$\Phi^{-t}(V_0\cap W^u_{{\cal Q},{\rm loc}}(z))\subset
V_2$.

That this requirement can be met follows from Theorem \ref{hodgenorm}
and the discussion in the proof of Lemma \ref{upperestimate}.

If $q\in \overline{\cal Q}-{\cal Q}$ then we choose 
closed neighborhoods 
\[V_0=\cup_i V(B_i^{ss},B_i^{su},t_0(1-\delta))\subset
V_3=\cup_iV(C_i^{ss},C_i^{su},t_0(1+\delta))\] of $q$ in 
$\overline{\cal Q}$ as in Proposition \ref{forwardnondiv}
such that $\cup_iB_i^j$ and $\cup_iC_i^{j}$ are the intersections
with $W_{{\cal Q},{\rm loc}}^{j}(q)$ of closed balls for the Hodge norm.
We require that property $(*)$ above holds true
(with a slight abuse of notation).

Let $u\in V_1$ and let
$r>0$ be such that $\Phi^ru=u$. 
Let $Y$ be the connected component containing $u$ of the
intersection $V_3\cap \Phi^r(V_2)$.
By the property $(*)$, we have 
$Y\supset  \Phi^r(V_1\cap W^{s}_{{\cal Q},{\rm loc}}(u))$. 
Moreover, the connected component containing $u$
of the intersection $V_3\cap \Phi^r(V_2\cap W^{u}_{{\cal Q},{\rm loc}}(u))$ 
contains the component containing $u$ of the intersection
$W^{u}_{\tilde {\cal Q},{\rm loc}}(u)\cap V_0$.
Thus
as in the proof of Lemma \ref{upperestimate}, we observe that
for any point $u\in V_0$ and every $r>0$ 
such that $\Phi^ru=u$
the Lebesgue measure of the intersection
$\Phi^rV_2\cap V_3$ is bounded from 
below by $e^{-hr}\chi$ where $\chi>0$ is a fixed
constant which only depends on $V_1,V_2,V_3$. 
Moreover, the number of periodic
points $z\in V_1$
of period $s\in [r-t_0,r+t_0]$ such that the
intersection components $\Phi^rV_2\cap V_3,\Phi^sV_2\cap V_3$ 
containing
$u,z$ are \emph{not} disjoint is bounded
from above by the cardinality of ${\rm Stab}(\tilde q)$ where
$\tilde q$ is a preimage of $q$ in ${\cal Q}^1(S)$..

For $q,z\in \overline{\cal Q}$ and $t>0$ write $q\approx_t z$ if
there are lifts $\tilde q,\tilde z$ of $q,z$ to ${\cal Q}^1(S)$ such that
\[d(P\Phi^s\tilde q,P\Phi^s\tilde z)< 1\text{ for }0\leq s\leq t.\]
Write moreover $q\sim_uz$ if there are lifts
$\tilde q,\tilde z$ of $q,z$ to ${\cal Q}^1(S)$ such that
\[d(\tilde q,\tilde z)< 1,d(P\Phi^u\tilde q,P\Phi^u\tilde z)< 1.\]
Note that if $y\approx_tz$ then also $y\sim_t z$. 
For a subset $D$ of $\overline{\cal Q}$ define
\begin{align}
U_t(D) &=\{z\mid z\approx_t y\text{ for some }y\in D\}\text{ and }\notag\\
Y_u(D)&=\{z\mid z\sim_u y\text{ for some }y\in D\}.\end{align}
Then $U_t(D)$ and $Y_u(D)$ are open neighborhoods of $D$.

For $j>0$ define 
\begin{align}
Z_{j}&=U_{j}(A\cap V_1)\cap V_1\text{ and }\\
W_{j,k}&=Y_k(Z_j)\cap V_1.\end{align}
Then for all $k>0,j>0$ 
each $j>0$, $Z_j$ is an open neighborhood of $A\cap V_1$ in 
$V_1$, and $W_{j,k}$ is an open neighborhood of 
$Z_j$ in $A\cap V_1$. Moreover, we have
$Z_{j}\supset Z_{j+1}$ for all $j$ and $\cap_jZ_{j}\supset A\cap V_1$.
If $z\in \cap_jZ_j-A$ then there is some $y\in A$ and there are
lifts $\tilde z,\tilde y$ of $z,y$ to ${\cal Q}^1(S)$ such that
$d(P\Phi^t(\tilde z),P\Phi^t(\tilde y))\leq 1$ for all $t\geq 0$.
However, up to removing from $\cap_jZ_j$ 
a set of vanishing Lebesgue measure,
this implies that $z\in W^{ss}_{{\cal Q},{\rm loc}}(y)$ \cite{M82,V86}. 
But $\lambda(A)=0$ and therefore
$\lambda(\cap_jZ_j)=\lambda(A\cap V_1)=0$ by absolute continuity.  
Since $\lambda$ is Borel regular,
the Lebesgue measures
of the sets $Z_{j}$ tend to zero as $j\to \infty$.

Similarly, we infer that
$\lambda(Z_j)=\lim\sup_{k\to \infty}\lambda(W_{j,k})$.
Thus 
for every $\kappa >0$ there are numbers $j_0=j_0(\kappa)>0$ and 
$k_0=k_0(\kappa)>j_0$ such that we have
$\lambda(W_{j,k})<\kappa$ for all $j\geq j_0,k\geq k_0$.
  
Now let $R>k_0+2\epsilon$ and let 
$w\in V_1\cap Z_{j_0}$ be a periodic point for $\Phi^t$ of 
prime period $r\in [R-\epsilon,R+\epsilon]$. Let $Z$ be the component
of $\Phi^rV_2\cap V_3$ containing $w$. Then 
every point in $Z$ is contained in $W_{j_0,R}$. 
By Lemma \ref{upperestimate} and its proof,
the Lebesgue
measure of this intersection component is bounded from
below by $\chi e^{-hR}$ where
$\chi>0$ is as above.
Moreover, the number of periodic points
$u\not=z$ for which these
intersection components are not
disjoint is uniformly bounded. 
In particular, there is a number $\beta>0$
not depending on $R,j_0$ such that
the number of these intersection
components is bounded from above by $\beta e^{hR}$ times the Lebesgue
measure of $W_{j_0,R}$, i.e. by $e^{hR}\beta\kappa$. This implies that 
we have $\mu(Z_{j_0})\leq \beta\kappa/2t_0$.
Since $\kappa>0$ was arbitrary we conclude that
$\mu(A\cap V_1)=0$.
Proposition \ref{absolute} follows.

\section{Proof of the theorem}

In this section we complete the proof of the theorem from
the introduction.
We continue to use the assumptions
and notations from Sections 2-5.

As before, let ${\cal Q}\subset {\cal Q}(S)$ 
be a component of a stratum, equipped
with the $\Phi^t$-invariant Lebesgue measure $\lambda$. 
Let ${\cal S}\subset {\cal Q}$ 
be the open dense $\Phi^t$-invariant subset of full Lebesgue measure of all
points $q$ with $\vert {\rm Stab}(q)\vert =1$. 
Then ${\cal S}$ is a manifold.

Let $q\in {\cal S}$ and let $U\subset {\cal S}$ be an open relative compact
contractible neighborhood of $q$.
For $n>0$ define
a \emph{periodic $(U,n)$-pseudo-orbit} for the
Teichm\"uller flow $\Phi^t$ on ${\cal Q}$
to consist of a point $x\in U$ and a number
$t\in [n,\infty)$ such that $\Phi^{t}x\in U$.
We denote such a periodic pseudo-orbit by $(x,t)$.
A periodic $(U,n)$-pseudo-orbit $(x,t)$ determines up to homotopy a
closed curve begin\-ning and en\-ding at $x$
which we call a \emph{characteristic curve} (compare
Section 4 of \cite{H07b}). This
characteristic curve is the concatenation of 
the orbit segment $\{\Phi^s x\mid 0\leq s\leq t\}$ 
with a smooth arc in $U$
which is parametrized on $[0,1]$ and
connects
the endpoint $\Phi^{t}x$ of the orbit segment 
with the starting point $x$.

Recall from Section 5 the definition of 
a recurrent point for the Teichm\"uller flow
on ${\cal Q}$.
Lemma 4.4 of \cite{H07b} shows

\begin{lemma}\label{quasigeodesic} There is a number 
$L>0$ and 
for every recurrent 
point $q\in {\cal S}$ there is 
an open relative compact contractible 
neighborhood $V$ of $q$ in ${\cal S}$ and there 
is a number $n_0>0$ depending on $V$
with the following property. Let $(x,t_0)$
be a periodic $(V,n_0)$-pseudo-orbit and
let $\gamma$ be a lift to ${\cal Q}^1(S)$ of a characteristic
curve of the pseudo-orbit. Then the curve
$t\to \Upsilon_{\cal T}(P\gamma(t))$ is an infinite unparametrized
$L$-quasi-geodesic in ${\cal C}(S)$.
\end{lemma}

{\bf Remark:} Lemma 4.4 of \cite{H07b} is formulated for
${\cal Q}(S)$ rather than for a component of a stratum. 
However, the statement and its proof immediately carry over to 
the result formulated in 
Lemma \ref{quasigeodesic}.

\bigskip

For a point $q\in {\cal Q}$ 
choose a preimage 
$\tilde q\in {\cal Q}^1(S)$ of 
$q$ and for $t>0$ define 
\[\beta(q,t)=d(\Upsilon_{\cal T}(P\tilde q),
\Upsilon_{\cal T}(P\Phi^t\tilde q)).\]
Note that $\beta(q,t)$ depends on the choice
of the map $\Upsilon_{\cal T}$ (and on the choice of the lift
$\tilde q$). However, by Lemma 3.3 of \cite{H10a}, there is a 
\emph{continuous} function $\tilde \beta:
{\cal Q}\times [0,\infty)\to \mathbb{R}$ and a number $a>0$ 
such that
$\vert \tilde \beta(q,t)-\beta(q,t)\vert \leq a$
for all $(q,t)$. 
In particular, the values 
$\lim\inf_{t\to \infty}\frac{1}{t}\beta(q,t)$ and 
$\lim\sup_{t\to\infty}\frac{1}{t}\beta(q,t)$ are independent of 
any choices made and coincide with
the corresponding values for $\tilde \beta$.
We use this observation to show

\begin{lemma}\label{asymptotic}
There is a number $c>0$ such that for 
$\lambda$-almost every $q\in {\cal Q}$ we have
\[\lim_{t\to \infty}\frac{1}{t}\beta(q,t)=c.\] 
\end{lemma}
\begin{proof}
It suffices to show the lemma for the continuous function
$\tilde \beta$. 

By the choice of $a>0$ and 
by the triangle inequality, we have
\[\tilde\beta(q,s+t)\leq \tilde \beta(q,s)+\tilde\beta(\Phi^sq,t)+3a\] for 
all $q\in {\cal Q},s,t\in \mathbb{R}$. 
Therefore the subadditive ergodic theorem shows that for 
$\lambda$-almost all $q\in {\cal Q}$ the limit
$\lim_{t\to\infty}\frac{1}{t}\tilde\beta(q,t) $ exists and 
is independent of $q$.
We are left with showing that this limit is positive.

By Lemma 2.4 of \cite{H10a}, there is a number
$r>0$ such that for every $z\in {\cal Q}^1(S)$ and all 
$t\geq s\geq 0$ we have
\[d(\Upsilon_{\cal T}(Pz),\Upsilon_{\cal T}(P\Phi^tz))\geq
d(\Upsilon_{\cal T}(Pz),\Upsilon_{\cal T}(P\Phi^sz))-r.\] 
Let $q\in {\cal Q}$ be a periodic
point for $\Phi^t$. 
Then there is a number $b>0$ such that 
for every lift $\tilde q$ of $q$ to ${\cal Q}^1(S)$ the 
map $t\to \Upsilon_{\cal T}(P\Phi^t\tilde q)$ is a 
biinfinite $b$-quasi-geodesic in ${\cal C}(S)$ \cite{H10a}.
Thus by inequality (\ref{distortion}) and
continuity of $\Phi^t$ we can find an
open neighborhood $U\subset {\cal Q}$ of $q$ and a number $T>0$
such that 
\[\tilde\beta(u,T)\geq 3r+3a \text{ for all } u\in U.\]
Now if $z\in {\cal Q}$ and if $n>k>0$ are such 
that the cardinality of the set of 
all numbers $i\leq n$ with $\Phi^{Ti}z\in U$ is not smaller than $k$ then 
$\tilde\beta(z,nT)\geq kr$.

The measure $\lambda$ is $\Phi^T$-invariant
and ergodic, and $\lambda(U)>0$. Thus  
by the Birkhoff ergodic
theorem,  
the proportion of time 
a typical orbit for the map $\Phi^T$ spends in $U$ 
is positive. 
The lemma follows.
\end{proof}

The next proposition is the main remaining 
step in the proof of the theorem from the introduction.

\begin{proposition}\label{lowerestimate}
For every recurrent point $q\in {\cal S}$, for every
neighborhood $V$ of $q$ in ${\cal S}$ 
and for every $\epsilon >0$ there is 
an open neighborhood $U\subset V$ of $q$ in ${\cal S}$
and a number $t_0>0$ such that
\[\lim\inf_{R\to\infty}H(U,R-t_0-\epsilon,R+t_0+\epsilon)e^{-hR}
\geq 2t_0\lambda(U)(1-\epsilon).\]
\end{proposition}
\begin{proof}
Let $q\in {\cal S}$ be recurrent  and
let $V$ be an open neighborhood of 
$q$ which satisfies the conclusion of Lemma \ref{quasigeodesic}
for some $n_0>0$. Let $\epsilon >0$. 
With the notations from Section 4, let 
$a_0<a_{\cal Q}(q),t_0<\min\{t_{\cal Q}(q),\log(1+\epsilon)/2h,\epsilon/4\}$ 
be such that $V_0=
V(B^{ss}_{{\cal Q}}(q,a_0),
B^{su}_{{\cal Q}}(q,a_0),t_0)\subset V$.
Choose a number $a_1<a_0$ which is sufficiently small that for
every $z\in V_1=V(B^{ss}_{\cal Q}(q,a_1),B^{su}_{\cal Q}(q,a_1),t_0)$ the
Jacobian at $z$ of the homeomorphism
\[V(B^{ss}_{\cal Q}(q,a_1),B^{su}_{\cal Q}(q,a_1),t_0)\to 
B^{ss}_{\cal Q}(q,a_1)\times B^{su}_{\cal Q}(q,a_1)\times [-t_0,t_0]\]
with respect to the measures $\lambda$ and $\lambda^{ss}\times 
\lambda^{su}\times dt$ is contained in the interval
$[(1+\epsilon)^{-1}, (1+\epsilon)]$. 
We may assume that any two points in a component  
$\tilde V_1$ of the preimage of $V_1$ 
can be connected in $\tilde V_1$ by a smooth curve whose
projection to ${\cal T}(S)$ is of length at most 
$\epsilon/2$.

Let $\alpha_0>0$ be as in Lemma \ref{expansion}. 
Let $\tilde q$ be a lift of $q$ to a component $\tilde {\cal Q}$
of the preimage of ${\cal Q}$ in ${\cal Q}^1(S)$.
Recall from Section 2
the definition of the map $F:\tilde {\cal A}\to
\partial{\cal C}(S)$.
Since $q$ is recurrent, the horizontal
and the vertical measured geodesic laminations of $\tilde q$ are
uniquely ergodic \cite{M82}. 
Let 
\[Z_1\subset Z_2\subset Z_3\subset V_1\] 
be neighborhoods of $q$ as in the proof of 
Lemma \ref{upperestimate} and let $\tilde Z_1\subset 
\tilde Z_2\subset \tilde Z_3\subset \tilde V_1$ be components of 
lifts of $Z_1\subset Z_2\subset Z_3\subset V_1$ 
to $\tilde {\cal Q}$ which contain $\tilde q$.  
These sets have the following property.
\begin{enumerate}
\item There are closed sets $K^i\subset C^i\subset W^i_{\cal Q}(q)$ with 
dense interior $(i=ss,su)$ 
and there are numbers $t_0>0,\delta >0$ such
$Z_1=V(K^{ss},K^{su},t_0),Z_2=V(K^{ss},C^{su},t_0),
Z_3=V(C^{ss},C^{su},t_0(1+\delta))$. 
\item For every $u\in \tilde Z_3$ 
the $\delta_{Pu}$-diameter of $F(\tilde Z_3\cap \tilde {\cal A})$ 
and of $F({\cal F}\tilde Z_3\cap \tilde {\cal A})$ 
is not bigger than $\alpha_0$.
\item $\lambda(Z_3)\leq \lambda(Z_1)(1+\epsilon)$.
\item There is a number $\rho>0$ with 
the following property. If $z\in \tilde Z_1$ and if 
$C\subset B^{su}_{\tilde {\cal Q}}(z,a_1)$
(or $C\subset B^{ss}_{\tilde {\cal Q}}(z,a_1)$)
is an open neighborhood of $z$  
such that the $\delta_{Pz}$-diameter
of $F(C\cap \tilde {\cal A})$ 
(or of $F({\cal F}(C)\cap \tilde {\cal A})$)
is not bigger than $\rho$ 
then $C\subset \tilde Z_3$ and the $\Phi^t$-orbit of every point of $C$
intersects
$\tilde Z_3$ in an arc of length $2t_0$.   
\end{enumerate}

Let $\Pi:\tilde {\cal Q}\to {\cal Q}$ be the
canonical projection. 
By Lemma \ref{asymptotic} and Lemma \ref{expansion}, 
there is a number $T>0$ and there is a Borel subset 
$Z_0\subset Z_1\cap \Pi (\tilde {\cal A})$ 
with 
\[\lambda(Z_0)>\lambda(Z_1)/(1+\epsilon)\] such
that for every $z\in \tilde Z_0=\tilde Z_1\cap \Pi^{-1}(Z_0)$ and every
$t\geq T$ we have
\[\delta_{Pz}\leq \rho\delta_{P\Phi^tz}/e^\kappa\text{ on }
D(\Phi^tz,\alpha_0)\]
where $\kappa >0$ is as in the estimate (\ref{deltacomparison}).
We may assume that $Z_0=V(A_0,K^{su},t_0)$ for
some Borel set $A_0\subset K^{ss}$.
In particular, 
we conclude as in the proof of Lemma \ref{upperestimate} 
(see the estimate (\ref{lambdaestimate}) that
(with some a-priori adjustment of the constant $\epsilon$) 
the following holds true. 
Let $z\in Z_0$ and let $t\geq T$ be such that 
$\Phi^tz\in Z_1$. Let $\hat E$ be the connected component
containing $\Phi^t z$ of the intersection $\Phi^tV_1\cap V_1$. 
Then the Lebesgue measure of the intersection 
$\Phi^tZ_2\cap Z_3\cap \hat E$ is not bigger than
\[e^{-ht}\lambda(Z_1)(1+\epsilon)^3\leq 
e^{-ht}\lambda(Z_0)(1+\epsilon)^4.\]

On the other hand, since the Lebesgue
measure is mixing, for sufficiently large $t>T$ we have
\[\lambda(\Phi^tZ_0\cap Z_0)\geq \lambda(Z_0)^2/(1+\epsilon).\]
Together 
this implies that the number of such intersection 
components is at least \[e^{ht}\lambda(Z_0)/(1+\epsilon)^5.\]

Next we claim that for sufficiently large $n\geq T$ 
and for a point $z\in Z_0$ with $\Phi^nz\in Z_1$
there is a periodic orbit for the flow $\Phi^t$ which 
intersects $Z_3$ in an arc of length at least $2t_0$ and whose
period is contained in the interval  
$[n-\epsilon,n+\epsilon]$.
To this end let 
$n_1>\max\{n_0,T\}$; then the 
conclusion of Lemma \ref{quasigeodesic} is satisfied for every
periodic $(Z_1,n_1)$-pseudo-orbit beginning 
at a point $z\in Z_0\subset V$.

Let $u\in Z_0$ be such that $\Phi^nu\in Z_1$ for 
some $n>n_1$.   
Let $\gamma$ be a characteristic curve of the 
periodic $(Z_1,n_1)$-pseudo-orbit $(u,n)$ which we obtain by
connecting $\Phi^nu\in Z_0$ with $u\in Z_0$ by a 
smooth arc contained in $Z_1$.
Up to replacing $n$ by $R=n+\tau$ for some $\tau\in [-2t_0,2t_0]
\subset [\epsilon/2,\epsilon/2]$ 
we may assume that $u\in V(K^{ss},K^{su}),\Phi^Ru\in V(K^{ss},K^{su})$. 

Let $\tilde \gamma$ be a lift of $\gamma$
to $\tilde {\cal Q}$ with starting point 
$\tilde \gamma(0)\in \tilde Z_0$. Then $\tilde \gamma$ is invariant
under a mapping class $g\in {\rm Mod}(S)$ whose conjugacy class
defines the homotopy class of $\gamma$ in ${\cal S}$. A fundamental
domain for the action of $g$ on $\tilde \gamma$ projects to 
a smooth arc in ${\cal T}(S)$ of length at most $R+\epsilon/2<n+\epsilon$.

By Lemma \ref{quasigeodesic} and the choice of
$Z_0,R$ the curve
$t\to \Upsilon_{\cal T}(P\tilde \gamma(t))$ is an 
unparametrized $L$-quasi-geodesic
in ${\cal C}(S)$ of infinite diameter. Up to perhaps a uniformly
bounded modification, 
this quasi-geodesic is invariant under the  mapping class 
$g\in {\rm Mod}(S)$, and $g$ acts on the 
quasi-geodesic $\Upsilon_{\cal T}(P\tilde \gamma)$ as a translation.
As a consequence,  $g$ 
acts on ${\cal C}(S)$ with unbounded orbits and hence it is pseudo-Anosov.
By invariance of $\tilde \gamma$ under $g$, 
the attracting fixed point of $g$ is just the endpoint of 
$\Upsilon_{\cal T}(P\tilde \gamma)$
in $\partial{\cal C}(S)$. 

Since $g$ is pseudo-Anosov, 
there is a closed orbit $\zeta$ for
$\Phi^t$ on ${\cal Q}(S)$ 
which is the projection of a $g$-invariant flow line $\tilde \zeta$ for
$\Phi^t$ in ${\cal Q}^1(S)$. 
The length of the orbit is at most $R+\epsilon$.
The image under the map $\Upsilon_{\cal T}P$ of the orbit 
$\tilde \zeta$ in ${\cal Q}^1(S)$ is an unparametrized
$p$-quasi-geodesic in ${\cal C}(S)$ which connects
the two fixed points for the action of $g$ on $\partial{\cal C}(S)$.

Assume that the characteristic curve $\gamma$  
is parametrized on 
$[0,R+1]$ with $\gamma(0)=u$. 
As in the proof of Theorem 4.3 of \cite{H07b},
we claim that for every $i>0$ we have
\[\delta_{P\tilde \gamma(0)}(F(\tilde \gamma(0)),F(\tilde
\gamma(iR+i)))< \rho\] (note that
this makes sense since the points $\tilde\gamma(iR+i)$ 
are lifts of recurrent points in ${\cal Q}(S)$ by assumption).
To see this we proceed by induction on $i$.
The case $i=1$ follows from the definition and from 
(\ref{deltacomparison}) above, so assume
that the claim is known for all $j\leq i-1$ and some $i\geq 0$.
By equivariance under the
action of the mapping class group 
we have
\begin{equation}
\delta_{P\tilde \gamma(R+1)}(F\tilde \gamma(R),F\tilde \gamma(R+1))
\leq e^\kappa\rho,
\end{equation}
moreover the distances
$\delta_{P\tilde \gamma(R)}, \delta_{P\tilde \gamma(R+1)}$ are
$e^\kappa$-bilipschitz equivalent.

Now $F(\tilde \gamma(jR+j))\in B_{\tilde
\gamma(R+1)}(F(\tilde \gamma(R+1)),\rho)$ 
for all $j\in \{1,\dots, i\}$ 
by the induction hypothesis and therefore
\begin{equation}
\delta_{P\tilde \gamma(R)}(F(\tilde \gamma(R)),
F(\tilde\gamma(jR+j)))\leq 2e^\kappa\rho.
\end{equation}
On the other hand, by the choice of
$\rho$ and the choice of $R$ and the fact that
$\tilde\gamma(0)\in \tilde Z_0$ 
we obtain that
\begin{equation}\delta_{P\tilde
\gamma(R)}(F\tilde \gamma(R), F\tilde \gamma(jR+j))\geq
\delta_{P\tilde \gamma(0)} (F\tilde \gamma(0),F\tilde
\gamma(jR+j))/2e^\kappa.\end{equation}
Together this implies the above  claim.

As a consequence, 
the attracting fixed point $\xi$ for the action 
of the pseudo-Anosov element $g$ on $\partial {\cal C}(S)$ 
is contained in the
ball $D(\tilde \gamma(0),\rho)$, moreover
it is contained in the closure of the set 
$F(W_{\tilde {\cal Q}}^{su}(\tilde q)\cap\tilde{\cal A})\subset
F(\tilde {\cal A}\cap \tilde {\cal Q})$.
The same argument also shows that
the repelling fixed point of $g$ is contained
in the intersection of 
$D(-\tilde \gamma(0),\rho)$ with
the closure of 
$F({\cal F}W_{\tilde {\cal Q}}^{ss}(\tilde q)\cap \tilde {\cal A})\subset
F(\tilde {\cal A}\cap \tilde {\cal Q})$. Since the map $F$
is closed we conclude that 
the axis of $g$ is contained in the
closure of $\tilde{\cal Q}$. Since $\tilde \gamma(0)\in Z_1$,
by property 4) above, this axis  
passes through the lift
$\tilde Z_3$ of $Z_3$ containing $\tilde q$. In other words,
the projection of this axis to $\overline{\cal Q}$ 
passes through $Z_3$, and, in particular, it is contained in ${\cal Q}$.
Moreover, it intersects the component of $\Phi^RZ_1\cap Z_3$ 
which contains $\Phi^Ru$.
As a consequence, 
the length of the axis is contained in 
$[R-\epsilon/2,R+\epsilon/2]\subset [n-\epsilon,n+\epsilon]$. 

To summarize, there is an injective assignment which
associates to every $R>n_0$ and to every
connected component of the intersection $\Phi^RZ_1\cap Z_1$ 
for $R>n_0>T$ which contains points in $\Phi^RZ_0\cap Z_0$ 
a subarc of length $2t_0$ of the intersection with $Z_3$ of 
a periodic orbit for $\Phi^t$ whose period is contained in
$[n-\epsilon,n+\epsilon]$. 
Together with the above discussion, this completes the
proof of the proposition.
\end{proof}

We use Proposition \ref{lowerestimate} 
to complete the proof of our theorem from the introduction.

\begin{theorem}\label{counting}
The Lebesgue measure on every stratum ${\cal Q}$ is obtained from
Bowen's construction.
\end{theorem}
\begin{proof}
By Proposition \ref{absolute} and Proposition \ref{lowerestimate}, 
it suffices to show the
following. Let $q\in {\cal Q}$ be birecurrent 
and let $\epsilon >0$. For $R>0$ let $\Gamma(R)$ be the
set of all periodic orbits of $\Phi^t$ in ${\cal Q}$ 
of period at most $R$. Then there is a 
compact neighborhood $K$ of $q$ in ${\cal Q}$ and there is a number
$n>0$ such that for every $N>n$ the measure 
\[\mu_N=e^{-hN}\sum_{\gamma\in \Gamma(R)}\delta(\gamma)\]
assigns the mass
\[\mu_N(K)\in 
[(1-\epsilon) \lambda(K), (1+\epsilon)\lambda(K)]\]
to $K$. However, this holds true by
Proposition \ref{absolute} and Proposition \ref{lowerestimate}.
This completes the proof of the theorem.
\end{proof}

{\bf Acknowledgement:} This work was carried out 
in fall 2007 while I participated
in the program on Teichm\"uller theory and
Kleinian groups at the MSRI in Berkeley. I thank
the organizers for inviting me to participate,
and I thank the MSRI for its hospitality.
I also thank Juan Souto for
raising the question which is answered in this note.

\bigskip

\noindent
MATHEMATISCHES INSTITUT DER UNIVERSIT\"AT BONN\\
ENDENICHER ALLEE 60, D-53115 BONN, GERMANY\\
%\smallskip
%\noindent
e-mail: ursula@math.uni-bonn.de

\end{document}